\documentclass[12pt]{amsart}
\usepackage{amsfonts,amsmath,amscd,amssymb}
\usepackage{mathrsfs}
\usepackage{latexsym,graphicx,verbatim,nameref}
\usepackage[numbers,sort&compress]{natbib}
\usepackage[all,cmtip]{xy}
\usepackage[hidelinks]{hyperref}
\usepackage{fancyhdr}
\def\Complex{\mathbb{C}}
\def\A{\mathcal{G}}
\def\Ar{\mathcal{G}_r)}
\def\cqg {compact quantum group}
\newtheorem{theorem}{Theorem}[section]
\newtheorem{corollary}[theorem]{Corollary}
\newtheorem{lemma}[theorem]{Lemma}
\newtheorem{proposition}[theorem]{Proposition}
\newtheorem*{theorem*}{Theorem}
\newtheorem*{proposition*}{Proposition}
\theoremstyle{definition}
\newtheorem{definition}[theorem]{Definition}
\newtheorem{remark}[theorem]{Remark}

\newtheorem{example}[theorem]{Example}

\newtheorem{conjecture}[theorem]{Conjecture}
\newcommand{\vv}{\vspace{4mm}\\}
\begin{document}
\title{Invariant subsets under compact quantum group actions}
\author{Huichi Huang}
\address{H. Huang, Mathematisches Institut, Einsteinstr. 62,  M\"unster, 48149, Germany}
\email{huanghuichi@uni-muenster.de}
\date{\today}
\begin{abstract}
 We investigate compact quantum group actions on unital $C^*$-algebras by analyzing invariant subsets and invariant states. In particular, we come up with the concept of compact quantum group orbits  and use it to show that countable compact metrizable spaces with infinitely many points are not quantum homogeneous spaces.
\end{abstract}

\maketitle
\tableofcontents

\section{Introduction}

A compact  quantum group is a unital $C^*$-algebra $A$ together with a unital $*$-homomorphism $\Delta:A\to A\otimes A$ satisfying the coassociativity $$(\Delta\otimes id)\Delta=(id\otimes \Delta)\Delta$$
and the cancelation laws that both $\Delta(A)(1\otimes A)$ and $\Delta(A)(A\otimes 1)$ are dense in $A\otimes A$. If $A$ is a commutative $C^*$-algebra, then $A=C(G)$ for some compact group $G$.  From the viewpoint of noncommutative topology $A=C(\mathcal{G})$ for some compact quantum space $\mathcal{G}$. So compact quantum groups are generalizations of compact groups. There are lots of similarities and differences between these two. For instance, firstly both of them both have the unique bi-invariant state called the Haar measure. But unlike the Haar measure of a compact group, the Haar measure of a compact quantum group need to be neither faithful nor tracial. Secondly, although there is a linear functional called the counit which plays the same role in a compact quantum group as the unit in a compact group,  the counit is only densely defined and not necessarily bounded.

An action of a compact quantum group $\mathcal{G}$ on a unital C*-algebra $B$ is a unital $*$-homomorphism $\alpha:B\rightarrow B\otimes A$ satisfying that
\begin{enumerate}
\item $(\alpha\otimes id)\alpha=(id\otimes \Delta)\alpha$;
\item $\alpha(B)(1\otimes A)$ is dense in $B\otimes A$.
\end{enumerate}
If $A=C(G)$ for some compact group $G$ and $B=C(X)$ for some compact Hausdorff space $X$, then the action $\alpha$ is just the action of $G$ on $X$ as homeomorphisms. Therefore actions of compact quantum groups on unital $C^*$-algebras are generalizations of compact groups on compact Hausdorff spaces. Moreover when a group acts on a space, the group elements are symmetries on the space. So when a compact quantum group $\mathcal{G}$ acts a unital $C^*$-algebra $B$, then $\mathcal{G}$ can be understood as a set of quantum symmetries of the compact quantum space $B$.

A compact quantum group action $\alpha$ of $\mathcal{G}$ on $B$ is called ergodic if $\{b\in B|\alpha(b)=b\otimes 1\}=\mathbb{C}$. If $\mathcal{G}$ is a compact group and $B=C(X)$ for a compact Hausdorff space $X$, then $\alpha$ is ergodic just means that the action is transitive. In this case $X$ is called homogeneous.  Generalizing the classical homogeneous space, we call a unital $C^*$-algebra $B$ a homogeneous space if $B$ admits an ergodic compact group action or a quantum homogeneous space if $B$ admits an ergodic compact quantum group action. Note that there are different definitions of quantum homogeneous spaces~(see~\cite{Vaes2005} for example) we adopt the one given by P. Podle\'{s} in~\cite[Definition 1.8]{Podles1995}.

A compact group is a compact quantum group, hence a homogeneous space is a quantum homogeneous space. However, the converse is not true.

It was shown by H\o egh-Krohn, Landstad and St\o rmer that a homogeneous space has a finite trace~\cite{HLS1981}. But the class of quantum homogeneous spaces includes operator algebras of some other types. For instance, S. Wang showed that some type III factors and Cuntz algebras are quantum homogeneous spaces~\cite{Wang1999}. So there exists compact quantum spaces which are quantum homogeneous space, but not homogeneous. Thus on some compact quantum spaces, namely Cuntz algebra, although there are no enough symmetries to make these spaces to be homogeneous spaces, there are enough quantum symmetries such that these spaces are quantum homogeneous space.

But when one considers compact quantum group actions on classical compact spaces, the situation is quite different. So far, all classical quantum homogeneous spaces  are  homogeneous spaces~\cite{Wang1998,Wang1999,BGS2011}. This means that on a classical compact space, if there are no enough symmetries, then there are no enough quantum symmetries.  This interesting phenomena leads us to conjecture that a compact Hausdorff space is a quantum homogeneous space if and only if it is a homogeneous space. Our main result in the paper is to confirm this conjecture in the case of compact Hausdorff spaces with countably infinitely many points.

\begin{theorem}\label{quantum homogenous space}
Any compact Hausdorff space with countably infinitely many points is not a quantum homogeneous space.
\end{theorem}

Theorem~\ref{quantum homogenous space} solves the conjecture for countable compact spaces.

To prove the main theorem, we use invariant subsets  and invariant states, formulate the concept of compact quantum group orbits and adopt them to study ergodic actions on compact spaces.

Our paper is organized as follows. In section 2  we collect some facts about compact quantum groups and their actions on unital $C^*$-algebras. In section 3, we derives some results about invariant subsets and invariant states which will be used later. Especially, we show that a compact quantum group action is ergodic iff there is a unique invariant state~(Theorem~\ref{ergodic and unique invariant measure}). Next we show that the ``support'' of an invariant state is an invariant subset~(Theorem~\ref{invariant support}) and show that as long as all invariant states are tricial or there exists a faithful tracial  invariant state, the compact quantum group is a Kac algebra~(Theorem~\ref{inv state is tracial haar measure is tracial}).  Section 4 is about compact quantum group actions on classical compact spaces.  We formulate the concept of orbits. Then we prove that an orbit is an invariant subset~(Theorem~\ref{invariance of orbit}) and that an action is ergodic iff there exists a unique orbit~(Theorem~\ref{a characterization of ergodic action by orbit}). Theorem~\ref{support is minimal} gives a characterization of   minimal invariant subsets and the relation between  minimal invariant subsets and orbits. Then we investigate actions on finite spaces and  prove Theorem~\ref{erdodic action on finite spaces} which will be used later to study orbits in spaces with countably points. In section 4.4 we prove Theorem~\ref{nonatomic inv measure} which says  the invariant measure on a quantum homogenous compact Hausdorff space with infinitely many points is non-atomic. In section 4.5, we apply results in previous subsections to study actions on compact spaces with countably infinitely many points and prove the main Theorem~\ref{quantum homogenous space}.

\section*{Acknowledgements}

I am grateful to Hanfeng Li for his long term support and encouragement. His advice is indispensable for the writing of this paper. Part of the paper was carried out during my staying in Chongqing University in the summer. I thank Dechao Zheng for his hospitality. The author also benefits from helpful discussions with Shuzhou Wang and would like to thank him for pointing out an error in Remark~\ref{rkac}. Lastly I express my gratitude to Piotr So{\l}tan for his comments and suggestions to the paper. The author currently is supported by ERC Advanced Grant No. 267079 and would like to thank it here.

\section{Preliminaries}

In this section, we recall some definitions and basic properties of compact quantum groups and their actions. We refer to \cite{Woronowicz1987, Woronowicz1998, MV1998} for basics of compact quantum groups and \cite{Podles1995,BOCA1995,LiH2009} for some background of compact quantum group actions.

Throughout this paper, for two unital C*-algebras $A$ and $B$, the notations $A\otimes B$ and $A\odot B$  stand for the minimal and the algebraic tensor product of $A$ and $B$ respectively.

For a $*$-homomorphism $\beta:B\rightarrow B\otimes A$, use $\beta(B)(1\otimes A)$  and $\beta(B)(B\otimes 1)$ to denote the linear span of the set $\{\beta(b)(1_B\otimes a)|b\in B, \,a\in A\}$ and  the linear span of the set $\{\beta(b_1)(b_2\otimes 1_A)|b_1, b_2\in B\}$ respectively.

For a C*-algebra $B$, we use $S(B)$ to denote the state space of $B$. For $\mu\in S(B)$, we denote $\{b\in B|\mu(b^*b)=0\}$ by $N_{\mu}$. If $N_{\mu}=\{0\}$, then $\mu$ is called faithful. If $\mu(ab)=\mu(ba)$ for all $a,b\in B$, then $\mu$ is called tracial.

 Let's first recall the definition of compact quantum group, which, briefly speaking, is the $C^*$-algebra of continuous functions on some compact quantum space with a group-like structure.

\begin{definition}[Definition~1.1 in~\cite{Woronowicz1998}]
A \textbf{compact quantum group} is a pair $(A,\Delta)$ consisting of a unital C*-algebra $A$ and a unital $*$-homomorphism $\Delta: A\rightarrow A\otimes A$ such that
\begin{enumerate}
\item $(id\otimes\Delta)\Delta=(\Delta\otimes id)\Delta$.
\item $\Delta(A)(1\otimes A)$ and $\Delta(A)(A\otimes 1)$ are dense in $A\otimes A$.
\end{enumerate}
\end{definition}

The $*$-homomorphism $\Delta$ is called the \textbf{coproduct} or \textbf{comultiplication} of $\A$. The first condition in the definition of compact quantum groups just says that the coproduct is associative, and the second condition says that the left cancellation law and the right cancellation law hold. Note that a compact semigroup in which cancellation laws hold is a group. Hence compact quantum groups are the quantum analogue of compact groups.

Furthermore, one can think of $A$ as $C(\mathcal{G})$, i.e., the C*-algebra of continuous functions on some quantum space $\mathcal{G}$ and in the rest of the paper we write  a compact quantum group $(A,\Delta)$ as $\mathcal{G}$.

There exists a unique state $h$ on $A$ such that $$(h\otimes id)\Delta(a)=(id\otimes h)\Delta(a)=h(a)1_A$$
 for all $a$ in $A$. The state $h$ is called the \textbf{Haar measure} of $\mathcal{G}$ or  the \textbf{Haar state} on $A$. Throughout this paper, we use $h$ to denote the Haar measure of $\mathcal{G}$.

\begin{example}[Examples of compact quantum groups]\label{examples of compact quantum groups}   \
\begin{enumerate}
\item For every non-singular $n \times n$ complex matrix $Q$ ($n > 1$),
 the universal compact quantum group $(A_u(Q), \Delta_Q)$~\cite[Theorem 1.3]{WV1996}
is generated by $u_{ij}$ ($i,j = 1 , \cdots, n$) with defining relations
(with $u = ( u_{ij} )$):
\vv
$ \; \;
u^* u = I_n = u u^*, \; \; \;
u^t Q {\bar u} Q^{-1} = I_n = Q {\bar u } Q^{-1} u^t$;
\vv
and the coproduct $\Delta_Q$ given by $\Delta_Q(u_{ij})=\sum_{k=1}^n u_{ik}\otimes u_{kj}$ for $1\leq i,j\leq n$. In particular, when $Q$ is the identity matrix, we denote $(A_u(Q), \Delta_Q)$ by $A_u(n)$.
\item The \textbf{ quantum permutation group} $(A_s(n),\Delta_n)$~\cite[Theorem 3.1]{Wang1998} is the universal $C^*$-algebra generated by $a_{ij}$ for $1\leq i,j\leq n$ under the relations
$$a_{ij}^*=a_{ij}=a_{ij}^2, \quad   \sum_{i=1}^n a_{ij}=\sum_{j=1}^n a_{ij}=1.$$
The coproduct $\Delta_n: A_s(n)\rightarrow A_s(n)\otimes A_s(n)$ is the $*$-homomorphism satisfying that
$$\Delta_n(a_{ij})=\sum_{k=1}^n a_{ik}\otimes a_{kj}.$$
\end{enumerate}
\end{example}

\begin{definition}
 Let $A$ be an associative $*$-algebra over $\Complex$ with an identity.  Assume
that $\Delta$ is a unital $*$-homomorphism from $A$ to
$A \odot A$ such that $(\Delta
\otimes id) \Delta = (id \otimes \Delta) \Delta$.  Also assume that there are
linear maps $\varepsilon : A \to \Complex$ and $\kappa : A \to A$ such that
$$(\varepsilon \otimes id)\Delta(a) = (id \otimes \varepsilon) \Delta(a) = a$$
$$m(\kappa \otimes id) \Delta(a) = m(id \otimes \kappa)\Delta(a) = \varepsilon(a)1$$
for all $a \in A$, where $m:A \odot A\to A$ is the multiplication map.
Then $(A, \Delta)$ is called a \textbf{Hopf $*$-algebra}~\cite[Definition 2.3]{MV1998}.
\end{definition}

A nondegenerate (unitary)~\textbf{representation} $U$ of  a compact quantum group $\mathcal{G}$ is an invertible~(unitary) element in $M(K(H)\otimes A)$ for some Hilbert space $H$ satisfying that $U_{12}U_{13}=(id\otimes \Delta)U$. Here $K(H)$ is the $C^*$-algebra of compact operators on $H$ and  $M(K(H)\otimes A)$ is the multiplier C*-algebra of  $K(H)\otimes A$.  We  write $U_{12}$ and
$U_{13}$ respectively for the images of  $U$ by two maps from $M(K(H)\otimes A)$ to $M(K(H)\otimes A\otimes A)$ where the first one is obtained by extending the map $x \mapsto x \otimes 1$ from $K(H) \otimes A$ to $K(H) \otimes A\otimes A$, and the second one is obtained by composing this map with the flip on the  last two factors. The Hilbert space $H$ is called the \textbf{carrier Hilbert space} of $U$. From now on, we always assume representations are nondegenerate. If the carrier Hilbert space $H$ is of finite dimension, then $U$ is called a finite dimensional representation of $\mathcal{G}$.

For two representations $U_1$ and $U_2$ with the carrier Hilbert spaces $H_1$ and $H_2$ respectively, the set of
 \textbf{intertwiners}  between $U_1$ and $U_2$, $Mor(U_1,U_2)$, is defined as
$$Mor(U_1,U_2)=\{T\in B(H_1,H_2)|(T\otimes 1)U_1=U_2(T\otimes 1)\}.$$
Two representations $U_1$ and $U_2$ are equivalent if there exists an invertible element $T$ in $Mor(U_1,U_2)$.
A representation $U$ is called \textbf{irreducible} if $Mor(U,U)\cong\Complex$.

Moreover, we have the following well-established facts about representations of compact quantum groups:
\begin{enumerate}
\item Every finite dimensional representation is equivalent to a unitary representation.
\item Every irreducible representation is  finite dimensional.
\end{enumerate}
Let $\widehat{\mathcal{G}}$ be the set of equivalence classes of irreducible representations of $\mathcal{G}$. For every $\gamma\in \widehat{\mathcal{G}}$, let $U^{\gamma}\in \gamma$  be unitary and $H_{\gamma}$ be its carrier Hilbert space with dimension $d_{\gamma}$. After fixing an orthonormal basis of $H_{\gamma}$, we can write $U^{\gamma}$ as  $(u^{\gamma}_{ij})_{1\leq i,j\leq d_{\gamma}}$ with $u^{\gamma}_{ij}\in A$. The matrix $\overline{U^{\gamma}}$ is still an irreducible representation~(not necessarily unitary) with the carrier Hilbert space $\overline{H_{\gamma}}$. It is called the \textbf{contragradient} representation of $U^\gamma$ and the equivalence class of $\overline{U^{\gamma}}$ is denoted by $\gamma^c$.
 There is a unique  positive invertible element $F^{\gamma}$ in $Mor(U^{\gamma},U^{\gamma^{cc}})$ such that $tr(F^{\gamma})=tr(F^{\gamma})^{-1}$. Denote $tr(F^{\gamma})$ by $M_{\gamma}$ and $M_{\gamma}$ is called the \textbf{ quantum dimension} of $\gamma$. Note that $F^{\gamma}>0$ is in $B(H_{\gamma})$ and can be expressed as a $d_{\gamma}\times d_{\gamma}$ matrix under the same orthonormal basis of $H_{\gamma}$ adopted by $U^{\gamma}$.

The linear space  $\mathscr{A}$ spanned by $\{u^{\gamma}_{ij}\}_{\gamma\in \widehat{\mathcal{G}},\,1\leq i,j\leq d_\gamma}$ is a Hopf $*$-algebra~\cite{Woronowicz1987,Woronowicz1998} such that
$$\Delta|_{\mathscr{A}}:\mathscr{A}\to \mathscr{A}\odot \mathscr{A},\qquad  \Delta(u^{\gamma}_{ij})=\sum_{m=1}^{d_{\gamma}}u^{\gamma}_{im}\otimes u^{\gamma}_{mj}.$$

Moreover, the following are true.
\begin{enumerate}
\item The Haar measure $h$ is \textbf{faithful} on $\mathscr{A}$, that is, if $h(a^*a)=0$ for an $a\in\mathscr{A}$, then $a=0$.
\item There exist uniquely a linear multiplicative functional $\varepsilon: \mathscr{A}\to \Complex$ and a linear antimultiplicative map $\kappa:\mathscr{A}\to\mathscr{A}$ such that
$$\varepsilon(u^{\gamma}_{ij})=\delta_{ij},\qquad \kappa(u^{\gamma}_{ij})=(u^{\gamma}_{ji})^*.$$ The two maps $\varepsilon$ and $\kappa$ are called the \textbf{counit} and the \textbf{antipodle} of $\mathcal{G}$ respectively.
\end{enumerate}

For $\gamma_1,\gamma_2\in \widehat{\mathcal{G}}$, $1\leq m,k\leq d_{\gamma_1}$ and $1\leq n,l\leq d_{\gamma_2}$, we have
\begin{equation}\label{eq:h1}
h(u^{\gamma_1}_{mk}u^{\gamma_2*}_{nl})=\frac{\delta_{\gamma_1\gamma_2}\delta_{mn}F^{\gamma_1}_{lk}}{M_{\gamma_1}},
\end{equation}
and
\begin{equation}\label{eq:h2}
h(u^{\gamma_1*}_{km}u^{\gamma_2}_{ln})=\frac{\delta_{\gamma_1\gamma_2}\delta_{mn}(F^{\gamma_1})^{-1}_{lk}}{M_{\gamma_1}}.
\end{equation}

A compact quantum group $(A',\Delta')$ is called a \textbf{quantum subgroup} of $\mathcal{G}$ if there exists a surjective $*$-homomorphism $\pi: A\to A'$ such that
$$(\pi\otimes \pi)\Delta=\Delta'\pi.$$ We can identify $A'$ with a quotient $C^*$-algebra of $A$, i.e., $A'\cong A/I$ for some ideal of $A$. We call the ideal $I$ a \textbf{Woronowicz $C^*$-ideal} of $A$. If we write $A'$ as $C(\mathcal{H})$ for some quantum space $\mathcal{H}$, we also call $\mathcal{H}$  a quantum subgroup of $\mathcal{G}$~\cite[Definition 2.13]{Wang1995}.

\begin{definition}[Definition 1.4 in \cite{Podles1995}]
An \textbf{action} of a compact quantum group $\mathcal{G}$ on a unital C*-algebra $B$ is a unital $*$-homomorphism $\alpha:B\rightarrow B\otimes A$ satisfying that
\begin{enumerate}
\item $(\alpha\otimes id)\alpha=(id\otimes \Delta)\alpha$;
\item $\alpha(B)(1\otimes A)$ is dense in $B\otimes A$.
\end{enumerate}
\end{definition}

An action $\alpha$ of a compact quantum group $\A$ on $B$ is called~\textbf{ergodic} if the fixed point algebra $B^{\alpha}=\{b\in B|\alpha(b)=b\otimes 1\}$ equals $\Complex1_B$.

Consider an action of $\mathcal{G}$ on $B$. For every $\gamma\in \widehat{\mathcal{G}}$, there is a  linear subspace $B_{\gamma}$ of $B$ with a basis $\mathscr S_{\gamma}=\{e_{\gamma ki}| k\in J_{\gamma}, 1\leq i\leq d_{\gamma}\}$ such that
$\alpha$ maps $B_{\gamma}$ into $B_{\gamma}\odot \mathscr{A}$  and $\alpha(e_{\gamma ki})=\sum_{j=1}^{d_{\gamma}}e_{\gamma kj}\otimes u^{\gamma}_{ji}$. Moreover $B_{\gamma}$ contains any other subspace of $B$ satisfying these two conditions.  The \textbf{quantum multiplicity} ${\rm mul}(B,\gamma)$ of $\gamma$ is defined as cardinality of $J_{\gamma}$, which does not depend on the choice of  $J_{\gamma}$~\cite[Thoerem 1.5]{Podles1995}. Moreover, $B_\gamma^*=B_{\gamma^c}$~\cite[Lemma 11]{BOCA1995}.
Hence  ${\rm mul}(B,\gamma)>0$ implies ${\rm mul}(B,\gamma^c)>0$.

Take
$\mathscr{B}=\bigoplus_{\gamma\in \widehat{\mathcal{G}}}B_{\gamma}$. It is known from~\cite[Thoerem 1.5]{Podles1995} that $\mathscr{B}$ is a dense $*$-subalgebra of $B$, which is called the \textbf{Podl\'{e}s algebra} of $B$. Also
$$\alpha|_{\mathscr{B}}:\mathscr{B}\to \mathscr{B}\odot \mathscr{A}, \qquad (id\otimes \varepsilon)\alpha|_{\mathscr{B}}=id_{\mathscr{B}}.$$

We say a bounded linear functional $\mu$ on $B$ is \textbf{$\alpha$-invariant} or briefly \textbf{invariant} if
$(\mu\otimes id)\alpha(b)=\mu(b)1_A$ for all $b\in B$. Denote by $Inv_{\alpha}$ the set of $\alpha$-invariant states on $B$.
It is known that $$Inv_{\alpha}=\{(\psi\otimes h)\alpha|\psi\in S(B)\}.$$

Suppose that  a compact quantum group $\mathcal{G}$ acts on $B_i$ by $\alpha_i$ for $i=1,2$. A unital $*$-homomorphism $s:B_1\rightarrow B_2$ is called
\textbf{equivariant} if
$$\alpha_2 \,s=(s\otimes id)\alpha_1.$$

Denote by $C(X)$ the $C^*$-algebra of complex-valued continuous functions on a compact Hausdorff space $X$. If a compact quantum group $\mathcal{G}$ acts on $B=C(X)$, then briefly we say that $\mathcal{G}$ acts on $X$.

\begin{definition}[Definition 1.8 in ~\cite{Podles1995}]
A unital $C^*$-algebra $B$ is called a \textbf{quantum homogeneous space} if $B$ admits an ergodic compact quantum group action.
\end{definition}
Briefly speaking, the investigation of actions of compact quantum groups on unital $C^*$-algebras is to study how compact quantum groups behave as symmetries of compact quantum spaces. Certainly there are many interesting examples of compact quantum group actions. Below we list some of them for later use, in particular, we give two examples of  compact quantum group actions on compact Hausdorff spaces.
\begin{example}[Examples of compact quantum group actions]\label{examples of CQG actions}  \
\begin{enumerate}
\item Every compact quantum group $\A$~acts on $A$ by the coproduct $\Delta$, and $\mathscr{A}$ is the Podl\'{e}s algebra of $A$.
\item The adjoint action $Ad_u$ of $(A_u(Q), \Delta_Q)$ on $M_n(\Complex)$ is given by
$$Ad_u(b)=u(b\otimes 1)u^*,$$ for every $b\in M_n(\Complex)$.
\item Recall that the Cuntz algebra $\mathcal{O}_n$~\cite{Cuntz1977} is the universal $C^*$-algebra generated by n($\geq 2$) isometries $S_1,S_2,...,S_n$ such that
$$\sum_{i=1}^n S_iS_i^*=1.$$
The compact quantum group $(A_u(Q), \Delta_Q)$  acts on $\mathcal{O}_n$ by $$\alpha(S_i)=\sum_{j=1}^n S_j\otimes u_{ji},$$ for $1\leq i\leq n$~\cite[Equation 5.2]{Wang1999}.
\item The quantum permutation group $A_s(n)$  acts on $X_n=\{x_1,x_2,\cdots,x_n\}$~\cite[Theorem 3.1]{Wang1998} by $$\alpha(e_i)=\sum_{j=1}^n e_j\otimes a_{ji},$$ where $e_i$ is the characteristic function  of $\{x_i\}$ for $1\leq i\leq n$.
\item  Let $Y$ be  a connected compact Hausdorff space and $Y_1$ is a closed subset of $Y$. Define an equivalence relation in $X_n\times Y$ as the following:
$(x_i,y)\sim (x_j,y)$ if $(x_i,y)=(x_j,y)$ or $y\in Y_1$. Then $A_s(n)$ acts on the connected compact space $X_n\times Y/\sim$ faithfully and the action $\alpha$ is given by
$$\alpha(\sum_{i=1}^n e_i\otimes f_i)=\sum_{i=1}^n\sum_{j=1}^n e_j\otimes f_i\otimes a_{ji}$$ for all $\sum_{i=1}^n e_i\otimes f_i\in C(X_n\times Y/\sim)$~\cite{Huang2013}.
  
\end{enumerate}
\end{example}

\section{Actions on compact quantum spaces}

\subsection{Faithful actions}

In this section, we give some equivalent conditions of faithful compact quantum group actions for future use. This is well known for experts, but for completeness and convenience, we give a proof here. Part of these results can be found in~\cite[Lemma 2.4]{Goswami2010}.

We first recall some definitions.
\begin{definition}[Definition 2.9 in ~\cite{Wang1995}]
For a compact quantum group $\mathcal{G}$, a unital C*-subalgebra $Q$ of $A$ is called a \textbf{compact quantum quotient group} of $\mathcal{G}$ if  $\Delta(Q)\subseteq Q\otimes Q$, and $\Delta(Q)(1\otimes Q)$ and $\Delta(Q)(Q\otimes 1)$ are dense in $Q\otimes Q$. That is, $(Q,\Delta|_{Q})$ is a compact quantum group. If $Q\neq A$, we call $Q$ a \textbf{proper compact quantum quotient group}.
\end{definition}

We say that  a \cqg~ action $\alpha$ on $B$ is \textbf{faithful} if there is no proper compact quantum quotient group $Q$ of $\A$ such that $\alpha$ induces an action $\alpha_q$ of $(Q,\Delta|_Q)$ on $B$ satisfying $\alpha(b)=\alpha_q(b)$ for all $b$ in $B$~\cite[Definition 2.4]{Wang1998}.

There are several equivalent descriptions of faithful actions.
\begin{proposition} \label{faithfulness}
Consider a \cqg~ action $\alpha$ of $\mathcal{G}$ on $B$. The following are equivalent:
\begin{enumerate}
\item The action $\alpha$ is faithful.
\item The $*$-subalgebra of $A$ generated by $(\omega\otimes id)\alpha(B)$ for all bounded linear functionals $\omega$ on $B$ is dense in $A$.
\item The $*$-subalgebra $\mathscr{A}_1$ of $\mathscr{A}$ generated by $(\omega\otimes id)\alpha(\mathscr{B})$ for all bounded linear functionals $\omega$ on $B$ is dense in $A$.
\item The $*$-subalgebra $\mathscr{A}_2$ of $\mathscr{A}$ generated by $u^{\gamma}_{ij}$ for all $\gamma\in \widehat{\mathcal{G}}$ and $1\leq i,j\leq d_{\gamma}$ such that ${\rm mul}(B,\gamma)>0$ is dense in $A$.
\item  $\mathscr{A}_2=\mathscr{A}$.
\end{enumerate}
\end{proposition}
\begin{proof}
$(2)\Rightarrow (1)$.
Suppose that the action $\alpha$ of $\mathcal{G}$ on $B$ induces an action $\alpha_q$ of  a quotient group $Q$ of $\mathcal{G}$  on $B$  such that $\alpha(b)=\alpha_q(b)$ for all $b$ in $B$. The $*$-subalgebra generated by $(\omega\otimes id)\alpha(B)$ for all bounded linear functional $\omega$ on $B$ is a subalgebra of $Q$. Hence $Q=A$ and $\alpha$ is faithful.

$(1)\Rightarrow (4)$.
Let $A_2$ be the closure of $\mathscr{A}_2$ in $A$.  We want to show that $(A_2,\Delta|_{A_2})$ is a quotient group of $\mathcal{G}$. First, since $\Delta({\mathscr{A}_2})\subseteq \mathscr{A}_2\odot \mathscr{A}_2$, we have that $\Delta(A_2) \subseteq A_2\otimes A_2$.

We next  show that $\Delta(A_2)(1\otimes A_2)$ is dense in $A_2\otimes A_2$. Since $u^{\gamma}$ is unitary for all $\gamma\in \widehat{\mathcal{G}}$ with ${\rm mul}(B,\gamma)>0$, we first have
$$\sum_{t=1}^{d_\gamma}\Delta(u^{\gamma}_{it})(1\otimes u^{\gamma*}_{jt})=u^{\gamma}_{ij}\otimes 1,$$  for all $1\leq i,j\leq d_{\gamma}$.
Note that $\sum_{t=1}^{d_\gamma}\Delta(u^{\gamma}_{it})(1\otimes u^{\gamma*}_{jt})$ belongs to $\Delta(A_2)(1\otimes A_2)$, so does $u^{\gamma}_{ij}\otimes 1$ for all $1\leq i,j\leq d_{\gamma}$. It follows that $$u^{\gamma_1}_{ij}u^{\gamma_2}_{kl}\otimes 1\in \Delta(A_2)(1\otimes A_2)(u^{\gamma_2}_{kl}\otimes 1)=\Delta(A_2)(u^{\gamma_2}_{kl}\otimes 1)(1\otimes A_2)\subseteq \Delta(A_2)(1\otimes A_2)$$ for all $\gamma_1,\gamma_2\in \widehat{\mathcal{G}}$ with positive multiplicity in $B$ and  all  $1\leq i,j\leq d_{\gamma_1}$ and $1\leq k,l\leq d_{\gamma_2}$. Inductively $u^{\gamma_1}_{i_1j_1}\cdots u^{\gamma_s}_{i_sj_s}\otimes 1\in\Delta(A_2)(1\otimes A_2)$ for all $\gamma_1,\cdots,\gamma_s\in \widehat{\mathcal{G}}$ with positive multiplicity in $B$ and  all  $1\leq i_t,j_t\leq d_{\gamma_t}$ with $1\leq t\leq s$.

Note that  $\mathscr{A}_2$ is the $*$-subalgebra of $\mathscr{A}$ generated by the matrix elements of $u^{\gamma}$ for all $\gamma\in \widehat{\mathcal{G}}$ with ${\rm mul}(B,\gamma)>0$. Also the adjoint of the matrix elements of  $u^{\gamma}$ are the matrix elements of  $u^{\gamma^c}$, the contragradient representation of $\gamma$.  Hence $\mathscr{A}_2$ is the subalgebra of $\mathscr{A}$ generated by the matrix elements of $u^{\gamma}$ for all $\gamma\in \widehat{\mathcal{G}}$ with positive  multiplicity in $B$. So $A_2\otimes 1$ is in the closure of $\Delta(A_2)(1\otimes A_2)$. Then for any $a,b\in A_2$, we have $a\otimes b=(a\otimes1)(1\otimes b)$ is in the closure of $\Delta(A_2)(1\otimes A_2)$ since $\Delta(A_2)(1\otimes A_2)(1\otimes b)\subseteq \Delta(A_2)(1\otimes A_2)$. Hence $\Delta(A_2)(1\otimes A_2)$ is dense in $A_2\otimes A_2$.

Similarly, we can prove that $1\otimes u^{\gamma *}_{ij}\in \Delta(A_2)(A_2\otimes 1)$ for all $\gamma\in \widehat{\mathcal{G}}$ with ${\rm mul}(B,\gamma)>0$ and all $1\leq i,j\leq d_{\gamma}$, and that $\Delta(A_2)(A_2\otimes 1)$ is dense in $A_2\otimes A_2$. Therefore, $A_2$ is a compact quantum quotient group of $A$. Next we show that $\alpha$ is an action of $(A_2,\Delta|_{A_2})$ on $B$.

Obviously $\alpha(B)\subseteq B\otimes A_2$. To show that $\alpha(B)(1\otimes A_2)$ is dense in $B\otimes A_2$, it is enough to prove that $e_{\gamma ki}\otimes 1\in \alpha(B)(1\otimes A_2)$ for all $\gamma\in \widehat{\mathcal{G}}$  such that ${\rm mul}(B,\gamma)>0$ and all $1\leq i\leq d_{\gamma}$ and  $1\leq k\leq {\rm mul}(B,\gamma)$. This follows from the following identity:
$$\sum_{t=1}^{d_\gamma}\alpha(e_{\gamma kt})(1\otimes u^{\gamma*}_{it})=e_{\gamma ki}\otimes 1.$$
Hence $\alpha$ is also an action of $A_2$ on $B$.  By the faithfulness of $\alpha$, we have that $A_2=A$.

$(3)\Leftrightarrow(4)$.
To prove the equivalence of (3) and (4), it suffices to show that $\mathscr{A}_1=\mathscr{A}_2$.
Obviously $\mathscr{A}_1\subseteq \mathscr{A}_2$. For  $\gamma\in \widehat{\mathcal{G}}$  such that ${\rm mul}(B,\gamma)>0$, we have that $\alpha(e_{\gamma ki})=\sum_{j=1}^{d_{\gamma}}e_{\gamma kj}\otimes u^{\gamma}_{ji}$ for $1\leq i\leq d_{\gamma}$ and  $1\leq k\leq {\rm mul}(B,\gamma)$. Note that $e_{\gamma ki}$'s are linearly independent. For every $1\leq s\leq d_{\gamma}$ and  every $1\leq l\leq {\rm mul}(B,\gamma)$, by the Hahn-Banach Theorem, there exists a bounded linear functional $\omega^{\gamma}_{ls}$ on $B$  such that $\omega^{\gamma}_{ls}(e_{\gamma ki})=\delta_{kl}\delta_{si}$ for $1\leq i\leq d_{\gamma}$ and  $1\leq k\leq {\rm mul}(B,\gamma)$. Therefore $(\omega^{\gamma}_{ks}\otimes id)\alpha(e_{\gamma ki})= u^{\gamma}_{si}\in \mathscr{A}_1$ for all $\gamma\in \widehat{\mathcal{G}}$  such that ${\rm mul}(B,\gamma)>0$, and for every $1\leq i\leq d_{\gamma}$ and  every $1\leq s\leq {\rm mul}(B,\gamma)$ . This implies that $\mathscr{A}_2\subseteq \mathscr{A}_1$, which proves the equivalence of (3) and (4).

$(2)\Leftrightarrow(3)$. The equivalence of (2) and (3) is immediate from the density of $\mathscr{B}$ in $B$ and the continuity of $(\omega\otimes id)\alpha$ for every bounded linear functional $\omega$ on $B$.

$(4)\Leftrightarrow(5)$. It is obvious that (5) implies (4). Now suppose that (4) is true. The $*$-subalgebra $\mathscr{A}_2$ is a Hopf $*$-subalgebra of $A$. A compact quantum group has a unique dense Hopf $*$-subalgebra~\cite[Theorem A.1]{BMT2001},  so (5) follows.
\end{proof}

\subsection{Invariant states}

In this subsection, we prove Theorem~\ref{l:T bijective} and Theorem~\ref{ergodic and unique invariant measure}.

First, for a compact quantum group, there is a reduced version of it in which the Haar measure is faithful~\cite[Theorem 2.1]{BMT2001}.

For a compact quantum group $\mathcal{G}$ with the Haar measure $h$ and the counit $\varepsilon$, let $N_h=\{a\in A|h(a^*a)=0\}$ and $\pi_r:A\to A/N_h$ be the quotient map. Then $N_h$ is a two-sided ideal of $A$~\cite[Proposition 7.9]{MV1998}. Furthermore, the following is true.
\begin{theorem*}[Theorem 2.1 of ~\cite{BMT2001}]
For a \cqg $\mathcal{G}$, the $C^*$-algebra $A_r=A/N_h$
is a compact quantum subgroup of $\A$ with
coproduct $\Delta_r$  determined by $\Delta_r (\pi_r(a))=(\pi_r\otimes \pi_r)\Delta(a)$, for all $a\in A$.
The Haar measure $h_r$ of $(A_r,\Delta_r)$  is given by $h=h_r\pi_r$ and $h_r$ is
faithful. Also, the quotient map $\pi_r$ is injective on~$\mathscr{A}$ and
the Hopf $*$-algebra of $(A_r,\Delta_r)$ is $\pi_r(\mathscr{A})$, with the counit $\varepsilon_r$ and the antipodle $\kappa_r$
determined by $\varepsilon= \varepsilon_r\pi_r$ and $ \pi_r \kappa = \kappa_r\pi_r$, respectively.
\end{theorem*}
\begin{definition}
The compact quantum group $(A_r,\Delta_r)$ is called the \textbf{reduced} compact quantum group of $\mathcal{G}$, and we write it as $\mathcal{G}_r$.
\end{definition}

From the theorem above, it is easy to check that any compact quantum group action of $\mathcal{G}$ on $B$ induces an action $\alpha_r$ of $(A_r,\Delta_r)$ on $B$ defined by
$$\alpha_r=(id\otimes\pi_r)\alpha.$$

Let $\alpha$ be an action of compact quantum group $\mathcal{G}$ on a unital C*-algebra $B$.
Let $B^{\alpha}=\{b\in B|\alpha(b)=b\otimes 1_A\}$. It is known that $$B^{\alpha}=(id\otimes h)\alpha(B). $$

Next we show that the space of invariant linear bounded functionals on $B$ is isometrically isomorphic to the dual space of $B^\alpha$.

Let $(B^{\alpha})'$ be the dual space of $B^{\alpha}$ and ${\rm Inv}(B)$ be the space of  $\alpha$-invariant bounded linear functionals on $B$. Define $T: {\rm Inv}(B)\rightarrow (B^\alpha)'$ as $T(\psi)=\psi|_{B^{\alpha}}$. Then

\begin{theorem}\label{l:T bijective}
The linear map $T$ is a bijective isometry.
\end{theorem}
\begin{proof}
Obviously $\|T\|\leq 1$, so $T$ is bounded.
Define the map $S: (B^\alpha)'\to  B'$ by $S(\varphi)=\widetilde{\varphi}$ for every $\varphi$ in $(B^{\alpha})'$
where $\widetilde{\varphi}$ is the linear functional on $B$ defined by $$\widetilde{\varphi}(b)=\varphi((id\otimes h)\alpha(b))$$
for every $b\in B$. Next we show that $S$ is the inverse of $T$.

First we show that $\widetilde{\varphi}$ is $\alpha$-invariant.
From $(\alpha\otimes id)\alpha=(id\otimes\Delta)\alpha$ and $(h\otimes id)\Delta=h(\cdot)1_A$, we have
\begin{align*}
 (\widetilde{\varphi}\otimes id)\alpha&=((\varphi\otimes h)\alpha\otimes id)\alpha=(\varphi\otimes h\otimes id)(\alpha\otimes id)\alpha \\
 &=(\varphi\otimes h\otimes id)(id\otimes\Delta)\alpha=(\varphi\otimes((h\otimes id)\Delta))\alpha\\
 &=(\varphi\otimes(h(\cdot)1_A))\alpha=\varphi((id\otimes h)\alpha(\cdot))1_A=\widetilde{\varphi}(\cdot)1_A.
\end{align*}
Hence $S$ maps $(B^\alpha)'$ into ${\rm Inv}(B)$. Moreover $\alpha(b)=b\otimes 1_A$  for any $b \in B^{\alpha}$. Hence
 $\widetilde{\varphi}(b)=\varphi(b)$ for any $b \in B^{\alpha}$. So  $\varphi$ is the restriction of $\widetilde{\varphi}$ on $B^{\alpha}$.
Therefore $\widetilde{\varphi}$ is $\alpha$-invariant and $TS(\varphi)=T(\widetilde{\varphi})=\varphi$. This shows the surjectivity of $T$.

Secondly for all $\phi\in {\rm Inv}(B)$ and all $b\in B$, we have $(\phi\otimes id)\alpha(b)=\phi(b)1_A$. Applying $h$ on both sides of the above equation, we get $(\phi\otimes h)\alpha(b)=\phi(b)$.  So
$$\widetilde{T(\phi)}(b)=T(\phi)((id\otimes h)\alpha(b))=\phi((id\otimes h)\alpha(b))=(\phi\otimes h)\alpha(b)=\phi(b)$$ for all $b\in B$. That is to say that $ST(\phi)=\widetilde{T(\phi)}=\phi$ for all  $\phi\in {\rm Inv}(B)$. Therefore $S$ is the inverse of $T$ and $T$ is bijective.

Moreover for every $\phi\in {\rm Inv}(B)$, we see that $\|T(\phi)\|\leq \|\phi\|$ and $\phi(b)=(\phi\otimes h)\alpha(b)=\phi((id\otimes h)\alpha(b))$ for each $b\in B$. If $b\in B$ and $\|b\|\leq1$, then $(id\otimes h)\alpha(b)\in B^\alpha$ and $\|(id\otimes h)\alpha(b)\|\leq 1$. So
\begin{align*}
\|\phi\|&=\sup_{\|b\|\leq1}|\phi(b)|=\sup_{\|b\|\leq1}|\phi((id\otimes h)\alpha(b))|    \\
       &=\sup_{\|b\|\leq1}|T(\phi)((id\otimes h)\alpha(b))|\leq \|T(\phi)\|.
\end{align*}
Therefore $\|T(\phi)\|=\|\phi\|$ for every $\phi\in {\rm Inv}(B)$ and $T$ is an isometry from ${\rm Inv}(B)$ onto $(B^\alpha)'$.
\end{proof}

The following theorem follows from Theorem~\ref{l:T bijective} immediately.

\begin{theorem}\label{ergodic and unique invariant measure}
A compact quantum group action $\alpha$ of $\A$ on $B$ is ergodic if and only if there is a unique $\alpha$-invariant state on $B$.
\end{theorem}

\begin{proof}
The ``only if'' part is well-known~\cite[Lemma 4]{BOCA1995}, and we just prove the ``if'' part.

Assume that there is a unique $\alpha$-invariant state on $B$. By Theorem~\ref{l:T bijective}, we have that ${\rm Inv}(B)\cong (B^\alpha)'$. So there is a unique state on $(B^\alpha)'$. Every bounded linear functional on $B^\alpha$ is a linear combination of states on $B^\alpha$, so $(B^\alpha)'=\Complex$. Hence $B^\alpha\subseteq (B^\alpha)''=\Complex$. Therefore $B^\alpha=\Complex$ and $\alpha$ is ergodic.
\end{proof}

For a compact quantum group action $\alpha$ of $\mathcal{G}$ on $B$, recall that the reduced action $\alpha_r$ of $\Ar$ on $B$ is defined by
 $$\alpha_r=(id\otimes \pi_r)\alpha.$$
 A state $\mu$ on $B$ is $\alpha$-invariant if and only if $\mu$ is $\alpha_r$-invariant since $(\mu\otimes h)\alpha=(\mu\otimes h_r)\alpha_r$. So by Theorem~\ref{ergodic and unique invariant measure}, the following is true.
\begin{corollary}\label{ergodic iff reduced ergodic}
A compact quantum group action $\alpha$ of $\A$ on $B$ is ergodic if and only if the reduced action $\alpha_r$ of $\Ar$ on $B$ is ergodic.
\end{corollary}

\subsection{Invariant subsets}

From now on, an ideal $I$ of a unital $C^*$-algebra $B$  always means a closed two-sided ideal, and we denote the quotient map from $B$ onto $B/I$ by $\pi_I$.
\begin{definition}
Suppose a compact quantum group $\mathcal{G}$ acts on $B$ by $\alpha$. An ideal $I$ of $B$ is called \textbf{$\alpha$-invariant} if for all $b\in I$, $$(\pi_I\otimes id)\alpha(b)=0.$$ 
 A proper $I$ is called \textbf{maximal} if any proper $\alpha$-invariant ideal $J\supseteq I$ of $B$ satisfies that $I=J$.
\end{definition}

\begin{remark}
If an ideal $I$ of $B$ is $\alpha$-invariant, then $\alpha$ induces an action $\alpha_I$ of $\mathcal{G}$ on $B/I$ given by
$$\alpha_I(b+I)=(\pi\otimes id)\alpha(b)$$  for all $b\in B$.

If $B=C(X)$ for a compact Hausdorff space $X$, then there is a one-one correspondence between closed subsets of $X$ and ideals of $B$. To say that an ideal is invariant under a compact group action is equivalent to say that the corresponding closed subset of $X$ is invariant. An ideal is maximal just says that the corresponding closed subset  is a minimal invariant subset of $X$.
\end{remark}

\begin{proposition}\label{invariance of invariant subset under reduced action}
If $I$ is an $\alpha$-invariant ideal of $B$, then $I$ is also $\alpha_r$-invariant.
\end{proposition}

\begin{proof}
Since $I$ is $\alpha$-invariant, we have that $(\pi_I\otimes id)\alpha(b)=0$ for all $b\in I$. Note that $\alpha_r=(id\otimes \pi_r)\alpha$. It follows that $(\pi_I\otimes id)\alpha_r(b)=(id\otimes\pi_r)(\pi_I\otimes id)\alpha(b)=0$ for all $b\in I$. So $I$ is also $\alpha_r$-invariant.
\end{proof}

If $I$ is an $\alpha$-invariant ideal of $B$, then for every $\alpha_I$-invariant state $\mu$ on $B/I$, the state $\mu \pi_I$ is an $\alpha$-invariant state on $B$. By Theorem~\ref{ergodic and unique invariant measure}, we have the following.

\begin{proposition}\label{inv subset is ergodic}
Consider a compact quantum group action $\alpha$ on $B$. If $I$ is an $\alpha$-invariant ideal of $B$ and $\alpha$ is ergodic, then $\alpha_I$ is also ergodic.
\end{proposition}

Take  an $\alpha$-invariant state $\mu$ on $B$. Let $\Phi_{\mu}:B\to B(H_{\mu})$ be the GNS representation of $B$ with respect to $\mu$ and denote $\ker{\Phi_{\mu}}$ by $I_{\mu}$. If $B$ is commutative, then
$$I_{\mu}=N_{\mu}=\{f\in B|\mu(f^*f)=0\}=\{f\in B|\,\,f|_{support\,\, of\, \mu}=0\}.$$
For a compact group action on a commutative $C^*$-algebra $B=C(X)$, the ideal $I_\mu$ is invariant is equivalent to that the support of $\mu$ is an invariant subset of $X$. The following theorem says that this is also true in the quantum case.

\begin{theorem}\label{invariant support}
Suppose that $\mathcal{G}$ acts on $B$ by $\alpha$ and $\mu$ is an $\alpha$-invariant state on $B$.  The ideal $I_{\mu}$ of $B$ is  $\alpha$-invariant, and the induced action on $B/I_{\mu}$, denoted by $\alpha_{\mu}$, is injective.
\end{theorem}
To prove Theorem~\ref{invariant support}, we need the following lemma:

\begin{lemma}\label{lemma of invariant support}
There exists an injective $*$-homomorphism $\beta:B(H_{\mu})\to L(H_{\mu}\otimes A)$ such that
$$\beta\Phi_{\mu}=(\Phi_{\mu}\otimes id)\alpha,$$
where $H_{\mu}\otimes A$ is the right Hilbert $A$-module with the inner product $\left<.,.\right>$ given by $\left<b_1\otimes a_1,b_2\otimes a_2\right>=\mu(b_1^*b_2)a_1^*a_2$ for $a_i\in A$ and $b_i\in B$, and $L(H_{\mu}\otimes A)$ is the set of adjointable maps on $H_{\mu}\otimes A$
\end{lemma}
\begin{proof}
We can define a bounded linear map $U: H_{\mu}\otimes A\to H_{\mu}\otimes A$ by
$$U(b\otimes a)=\alpha(b)(1\otimes a),$$ for all $b\in B$ and $a\in A$.

Using the argument in~\cite[Lemma 5]{Boca1995}, we get that $U$ is a unitary representation of $\mathcal{G}$ with the carrier Hilbert space $H_{\mu}$.

Let $\beta(T)=U(T\otimes 1)U^*$ for $T\in B(H_{\mu})$. It is easy to see that $\beta$ is an injective $*$-homomorphism from  $B(H_{\mu})$ into $L(H_{\mu}\otimes A)$. To prove $\beta\Phi_{\mu}=(\Phi_{\mu}\otimes id)\alpha$, it is enough to show that
$$\beta\Phi_{\mu}(b)(\alpha(b_1)(1\otimes a_1))=(\Phi_{\mu}\otimes id)\alpha(b)(\alpha(b_1)(1\otimes a_1))$$ for all $a_1\in A$ and $b,b_1\in B$, since $\alpha(B)(1\otimes A)$ is dense in $B\otimes A$.
From the definitions of $U$ and $\beta$ and that $U$ is unitary,
\begin{equation*}
\begin{split}
\beta\Phi_{\mu}(b)(\alpha(b_1)(1\otimes a_1))
&=u(b\otimes 1)u^*(\alpha(b_1)(1\otimes a_1))  \\
&=u(bb_1\otimes a_1)=\alpha(bb_1)(1\otimes a_1).   \notag
\end{split}
\end{equation*}
On the other hand, we have that
$$(\Phi_{\mu}\otimes id)\alpha(b)(\alpha(b_1)(1\otimes a_1))=\alpha(bb_1)(1\otimes a_1).$$

This completes the proof.
\end{proof}
Now we are ready to prove Theorem~\ref{invariant support}.
\begin{proof}
By Lemma~\ref{lemma of invariant support}, we have that $\beta\Phi_{\mu}=(\Phi_{\mu}\otimes id)\alpha$. Hence $(\Phi_{\mu}\otimes id)\alpha(b)=\beta\Phi_{\mu}(b)=0$ for any $b\in I_{\mu}$. Let $\pi_{\mu}$ be the quotient map from $B$ onto $B/I_{\mu}$ and $\widehat{\Phi_{\mu}}$ be the injective $*$-homomorphism from $B/I_{\mu}$ into $B(H_{\mu})$ induced by $\Phi_{\mu}$, then
$$\Phi_{\mu}=\widehat{\Phi_{\mu}}\pi_{\mu}.$$ The injectivity of $\widehat{\Phi_{\mu}}$ gives us the injectivity of $\widehat{\Phi_{\mu}}\otimes id$. So for $b\in I_{\mu}$, the identities $$0=\beta\Phi_{\mu}(b)=(\Phi_{\mu}\otimes id)\alpha(b)=(\widehat{\Phi_{\mu}}\otimes id)(\pi_{\mu}\otimes id)\alpha(b)$$  implies that $(\pi_{\mu}\otimes id)\alpha(b)=0$, which proves the invariance of $I_{\mu}$.

If $\alpha_{\mu}(b+I_{\mu})=0$ for some $b\in B$, then $(\pi_{\mu}\otimes id)\alpha(b)=0$. Hence $(\widehat{\Phi_{\mu}}\otimes id)(\pi_{\mu}\otimes id)\alpha(b)=0$. Then it follows from $\Phi_{\mu}=\widehat{\Phi_{\mu}}\pi_{\mu}$ that $(\Phi_{\mu}\otimes id)\alpha(b)=0$. Since $\beta\Phi_{\mu}=(\Phi_{\mu}\otimes id)\alpha$, we have that $\beta\Phi_{\mu}(b)=0$. That is to say $\beta\widehat{\Phi_{\mu}}\pi_{\mu}(b)=0$. Since $\beta$ and $\widehat{\Phi_{\mu}}$ are both injective, we have that $\pi_{\mu}(b)=0$, which proves the injectivity of $\alpha_{\mu}$.
\end{proof}

Theorem~\ref{invariant support} shows that there always exist invariant subsets such that the induced action is injective. Furthermore, there exists the ``largest'' invariant subsets with injective action.

\begin{theorem}\label{largest inv subset admitting injective action}
Consider a compact quantum group action $\alpha$ on $B$. Let $\mathfrak{J}$  be the set of all proper~$\alpha$-invariant ideals of $B$ with injective induced actions. Denote the ideal $\bigcap_{I\in\mathfrak{J}}I$ by $I_{\mathfrak{J}}$. Then $I_{\mathfrak{J}}$ is also $\alpha$-invariant. The induced action on $B/\bigcap_{I\in\mathfrak{J}}I$ is injective and satisfies the following universal property:

for any proper invariant ideal $I$ with injective induced action on $B/I$, there is an equivariant map from $B/I_{\mathfrak{J}}$ onto $B/I$.
\end{theorem}

We first prove two preliminary results.
The first one says that the closure of the union of invariant subsets is still invariant.

\begin{proposition}\label{union is invariant}
Consider a compact quantum group action $\alpha$ of $\A$ on $B$. If $I_{\lambda}$ is an  $\alpha$-invariant ideal of $B$ for every $\lambda\in \Lambda$, then $\bigcap_{\lambda\in \Lambda}I_{\lambda}$ is also an $\alpha$-invariant ideal of $B$.
\end{proposition}

Before we start to prove the above Proposition~\ref{union is invariant}, we prove the following result which is well-known for experts. For the convenience of the reader,  we give a proof below.

We need the following proposition.
\begin{proposition}\cite[Proposition 3.2.11]{LiB1992}
Let $A_i$ be a unital $C^*$-algebra for every $1\leq i\leq n$. Then the space $A_1'\odot A_2'\cdots \odot A_n' $ of the algebraic tensor product of the dual spaces $A_i'$ is $weak^*$ dense in the dual space $(A_1\otimes A_2\cdots \otimes A_n)'$ of $A_1\otimes A_2\cdots \otimes A_n$.
\end{proposition}
\begin{lemma}\label{tensor states}
For two unital $C^*$-algebras $A$ and $B$, the set $\mathcal{T}=\{\varphi\otimes \psi|\varphi\in S(A),\,\psi\in S(B)\}$ separates the points of $A\otimes B$.
\end{lemma}
\begin{proof}
From the above proposition, $A'\odot B'$ separates the points of $A\otimes B$. Note that every element of $A'\otimes B'$ is a linear combination of some elements in $\mathcal{T}$, so  $\mathcal{T}$ separates the points of $A\otimes B$.
\end{proof}

\begin{proof}[Proof of Proposition~\ref{union is invariant}]\
Take any $b\in \bigcap_{\lambda}I_{\lambda}$ and any $\psi\in S(A)$. By the invariance of $I_{\lambda}$, we have that $\pi_{\lambda}((id\otimes \psi)\alpha(b))=0$ for all $\lambda\in \Lambda$.
Hence  for any $\psi\in S(A)$, we have that $(id\otimes \psi)\alpha(b)\in \bigcap_{\lambda}I_{\lambda}$. Denote the quotient map from $B$ onto $B/\bigcap_{\lambda\in \Lambda}I_{\lambda}$ by $\pi_{\Lambda}$. Then we get  $(\pi_{\Lambda}\otimes\psi)\alpha(b)=0$ for all  $\psi\in S(A)$. It follows that $(\phi\otimes\psi)((\pi_{\Lambda}\otimes id)\alpha(b))=0$ for all $\phi\in S(B/\bigcap_{\lambda\in \Lambda}I_{\lambda})$ and  all  $\psi\in S(A)$. By Lemma~\ref{tensor states}
the set $\{\phi\otimes\psi\}_{\phi\in S(B/\bigcap_{\lambda\in \Lambda}I_{\lambda}),\psi\in S(A)}$ separates the points of $(B/\bigcap_{\lambda\in \Lambda}I_{\lambda})\otimes A$.
So $(\pi_{\Lambda}\otimes id)\alpha(b)=0$ and $\bigcap_{\lambda}I_{\lambda}$ is invariant.
\end{proof}

\begin{lemma}\label{inj lemma}
Consider a compact quantum group action $\alpha$ of $\A$ on $B$. Suppose that for every $\lambda\in \Lambda$, the ideal $I_{\lambda}$ is a proper $\alpha$-invariant ideal of $B$ such that the induced actions on $B/I_{\lambda}$ is injective. Denote $\bigcap_{\lambda\in\Lambda}I_{\lambda}$ by $I_{\Lambda}$. The ideal $I_{\Lambda}$ is also an $\alpha$-invariant ideal and the induced action on $B/I_{\Lambda}$ is injective. Furthermore, the quotient map from $B/I_{\Lambda}$ onto $B/I_{\lambda}$ is equivariant for every $\lambda\in \Lambda$.
\end{lemma}
\begin{proof}
The invariance of $I_{\Lambda}$ follows from Proposition~\ref{union is invariant}. Denote the induced action on $B/I_{\Lambda}$ by $\alpha_{\Lambda}$   and the quotient maps from $B$ onto $B/I_{\lambda}$  and $B/I_{\Lambda}$ by $\pi_{\lambda}$ and $\pi_{\Lambda}$ respectively. For every $B/I_{\lambda}$, we denote the canonical surjective map from $B/I_{\Lambda}$ onto $B/I_{\lambda}$ by $\pi_{\lambda\Lambda}$. Since $\pi_{\lambda\Lambda}\pi_{\Lambda}=\pi_\lambda$ and both $\pi_{\lambda}$ and $\pi_{\Lambda}$ are equivariant, we have the commutative diagram:
\[
\begin{CD}
    B/I_{\Lambda}  @>\alpha_{\Lambda}>> B/I_{\Lambda}\otimes A\\
    @V \pi_{\lambda\Lambda} VV  @V \pi_{\lambda\Lambda}\otimes id VV\\
    B/I_{\lambda}  @>\alpha_{\lambda}>>  B/I_{\lambda}\otimes A
\end{CD}
\]
This proves the equivariance of $\pi_{\lambda\Lambda}$. Moreover,  if $\alpha_{\Lambda}(b+I_{\Lambda})=0$ for some $b\in B$, then
$$0=(\pi_{\lambda\Lambda}\otimes id)\alpha_{\Lambda}(b+I_{\Lambda})=\alpha_{\lambda}\pi_{\lambda\Lambda}(b+I_{\Lambda}).$$
Note that $\alpha_{\lambda}$ is injective, thus $b+I_{\lambda}=\pi_{\lambda\Lambda}(b+I_{\Lambda})=0$ for all $\lambda\in \Lambda$. Therefore $b\in I_{\Lambda}$, which proves the injectivity of $\alpha_{\Lambda}$.
\end{proof}
Theorem~\ref{largest inv subset admitting injective action} follows from Lemma~\ref{inj lemma} immediately.

\begin{example}[Examples of invariant ideals]\label{ex:Inv sub}
\
\begin{enumerate}
\item Consider that a compact quantum group $\mathcal{G}$ acts on $A$ by $\Delta$. The Haar measure $h$ is the unique $\Delta$-invariant state on $A$. Since $N_{h}$ is an ideal~\cite[Proposition 7.9]{MV1998}, we have that $I_h=N_h$. Hence $N_h$ is an invariant ideal of $A$.
\item If $B$ is commutative, then $N_{\mu}=I_{\mu}$ for every $\alpha$-invariant state $\mu$ on $B$ and $N_{\mu}$ is an $\alpha$-invariant ideal of $B$ by Theorem~\ref{invariant support}.
\end{enumerate}
\end{example}

\subsection{Kac algebra and tracial invariant states}

\begin{definition}
A compact quantum group $\mathcal{G}$ is called a \textbf{Kac algebra} if one of the following equivalent conditions holds~\cite[Theorem 1.5]{Woronowicz1998}~\cite[Example 1.1]{Banica1999-2}~\cite[Definition 8.1]{Banica1999}:
\begin{enumerate}
\item The Haar measure $h$ of $\mathcal{G}$ is tracial.
\item The antipode $\kappa$ of $\mathcal{G}$ satisfies that $\kappa^2=id$ on $\mathscr{A}$.
\item $F^{\gamma}=id$ for all $\gamma\in \widehat{\mathcal{G}}$.
\end{enumerate}
\end{definition}

For an ergodic action $\alpha$ of  a compact quantum group $\mathcal{G}$ on $B$, in general, the unique $\alpha$-invariant state $\mu$ on $B$ is not necessarily tracial~(See Remark 3.34 below). In~\cite{Goswami2010}, Goswami showed that if $\mathcal{G}$ acts on a unital $C^*$-algebra $B$ ergodically and faithfully, and the unique $\alpha$-invariant state $\mu$ on $B$ is tracial, then $\mathcal{G}$ is a Kac algebra~\cite[Theorem 3.2]{Goswami2010}. Actually Goswami proved this result with the assumption that $B$ is commutative, but his proof works in the noncommutative case with the assumption of the traciality of $\mu$.

Using a different method, we generalize this result to faithful ~(not necessarily ergodic) actions, and show that traciality  of $h$ depends on traciality of invariant states~(see Theorem~\ref{inv state is tracial haar measure is tracial} below).

\begin{lemma}\label{tracial lemma}
Suppose that $\mathcal{G}$ acts on $B$ by $\alpha$. Take $\gamma\in\widehat{\mathcal{G}}$ such that ${\rm mul}(B,\gamma)>0$. If there exists a state $\varphi$ on $B$ satisfying that $$\varphi(\sum_{1\leq s\leq d_{\gamma}}e_{\gamma ks}e_{\gamma ks}^*)>0$$ and $(\varphi\otimes h)\alpha(e_{\gamma kj}e_{\gamma ki}^*)=(\varphi\otimes h)\alpha(e_{\gamma ki}^*e_{\gamma kj})$ for some $1\leq k\leq {\rm mul}(B,\gamma)$ and all $1\leq i,j\leq d_{\gamma}$, then $F^{\gamma}=id$.
\end{lemma}
\begin{proof}
For convenience, in the proof we denote $F^\gamma$ by $F$ for $\gamma\in \widehat{\mathcal{G}}$. Recall that for $\gamma_1,\gamma_2\in \widehat{\mathcal{G}}$, $1\leq m,k\leq d_{\gamma_1}$ and $1\leq n,l\leq d_{\gamma_2}$, we have that
\begin{equation*}
h(u^{\gamma_1}_{mk}u^{\gamma_2*}_{nl})=\frac{\delta_{\gamma_1\gamma_2}\delta_{mn}F_{lk}}{M_{\gamma_1}},
\end{equation*}
and
\begin{equation*}
h(u^{\gamma_1*}_{km}u^{\gamma_2}_{ln})=\frac{\delta_{\gamma_1\gamma_2}\delta_{mn}(F^{-1})_{lk}}{M_{\gamma_1}}.
\end{equation*}
Hence
\begin{equation*}
\begin{split}
&(\varphi\otimes h)\alpha(e_{\gamma kj}e_{\gamma ki}^*)                                                      \\
&=\sum_{1\leq s,t\leq d_{\gamma}}\varphi(e_{\gamma ks}e_{\gamma kt}^*)h(u^{\gamma}_{sj}(u^{\gamma}_{ti})^*)   \\
&=\sum_{1\leq s,t\leq d_{\gamma}}\varphi(e_{\gamma ks}e_{\gamma kt}^*)\delta_{st}
\frac{F_{ij}}{M_{\gamma}}         \\
&=\sum_{1\leq s\leq d_{\gamma}}\varphi(e_{\gamma ks}e_{\gamma ks}^*)\frac{F_{ij}}{M_{\gamma}},  \notag             \end{split}
\end{equation*}

and
\begin{equation*}
\begin{split}
&(\varphi\otimes h)\alpha(e_{\gamma ki}^*e_{\gamma kj})                                                      \\
&=\sum_{1\leq s,t\leq d_{\gamma}}\varphi(e_{\gamma ks}^*e_{\gamma kt})h((u^{\gamma}_{si})^*u^{\gamma}_{tj})   \\
&=\sum_{1\leq s,t\leq d_{\gamma}}\varphi(e_{\gamma ks}^*e_{\gamma kt})
(F^{-1})_{ts}\frac{\delta_{ij}}{M_{\gamma}}.            \notag
\end{split}
\end{equation*}

From $(\varphi\otimes h)\alpha(e_{\gamma kj}e_{\gamma ki}^*)=(\varphi\otimes h)\alpha(e_{\gamma ki}^*e_{\gamma kj})$ and $\sum_{1\leq s\leq d_{\gamma}}\varphi(e_{\gamma ks}e_{\gamma ks}^*)>0$, we have that
$$F_{ij}=\frac{\sum_{1\leq s,t\leq d_{\gamma}}\varphi(e_{\gamma ks}^*e_{\gamma kt})
(F^{-1})_{ts}\delta_{ij}}{\sum_{1\leq s\leq d_{\gamma}}\varphi(e_{\gamma ks}e_{\gamma ks}^*)},$$
which implies that $F$ is a scalar matrix under a fixed orthonormal basis of $H_{\gamma}$.
 Note that $tr(F)=tr(F^{-1})$, hence we get $F=I$ under a fixed orthonormal basis of $H_{\gamma}$, which means that $F=id$.
\end{proof}

\begin{proposition}\label{tracial inv states imply tracial haar measure}
Suppose that a compact quantum group $\mathcal{G}$ acts on $B$ by $\alpha$. If one of the following two conditions is true:
  \begin{enumerate}
  \item every invariant state on $B$ is tracial,
  \item there exists a faithful  tracial invariant state,
  \end{enumerate}
  then for all  $\gamma\in\widehat{\mathcal{G}}$ such that ${\rm mul}(B,\gamma)>0$, we have that $F^{\gamma}=id$.
\end{proposition}
\begin{proof}
 Suppose that every invariant state on $B$ is tracial. Note that $(\varphi\otimes h)\alpha$ is an $\alpha$-invariant state for any $\varphi\in S(B)$. By assumption $(\varphi\otimes h)\alpha$ is tracial. For any $\gamma\in\widehat{\mathcal{G}}$ with ${\rm mul}(B,\gamma)>0$, since for any $1\leq k\leq {\rm mul}(B,\gamma)$, $\sum_{1\leq s\leq d_{\gamma}}e_{\gamma ks}e_{\gamma ks}^*>0$, there exists a $\varphi_{\gamma}\in S(B)$ satisfying that $\sum_{1\leq s\leq d_{\gamma}}\varphi_{\gamma}(e_{\gamma ks}e_{\gamma ks}^*)>0$. Hence by Lemma~\ref{tracial lemma} we have that $F^{\gamma}=id$.

On the other hand, if there exists a faithful  tracial invariant state on $B$, say $\psi$, then $(\psi\otimes h)\alpha=\psi$ and $\psi$ satisfies the conditions of Lemma~\ref{tracial lemma}. Hence $F^{\gamma}=id$ for all  $\gamma\in\widehat{\mathcal{G}}$ with positive ${\rm mul}(B,\gamma)$.
\end{proof}
\begin{remark}
A special case of Proposition~\ref{tracial inv states imply tracial haar measure} is the following:

 If $\alpha$ is ergodic and  the unique $\alpha$-invariant state $\mu$ is tracial, then  for all  $\gamma\in\widehat{\mathcal{G}}$ such that ${\rm mul}(B,\gamma)>0$, we have that $F^{\gamma}=id$.

A slightly different version of this result appears in~\cite[Theorem 3.1]{Pinzari2012} where a necessary and sufficient condition of traciality of the unique invariant state of an ergodic action is given.
\end{remark}

\begin{theorem}\label{inv state is tracial haar measure is tracial}
Suppose that a compact quantum group~ $\mathcal{G}$ acts on $B$ by $\alpha$ faithfully. If one of the following two conditions is true:
  \begin{enumerate}
  \item every invariant state on $B$ is tracial,
  \item there exists a faithful  tracial invariant state on $B$,
  \end{enumerate}
then $\mathcal{G}$ is  a Kac algebra.
\end{theorem}
\begin{proof}
 Note that for all $\gamma\in\widehat{\mathcal{G}}$ and a unitary $u^{\gamma}\in \gamma$,  it follows from ~\cite[Theorem 5.4]{Woronowicz1987} that $(id\otimes\kappa^2)u^{\gamma}=F^{\gamma}u^{\gamma}(F^{\gamma})^{-1}$. By Proposition~\ref{tracial inv states imply tracial haar measure}, we see that  $F^{\gamma}=id$ for all $\gamma\in\widehat{\mathcal{G}}$ such that ${\rm mul}(B,\gamma)>0$. So
 $$\kappa^2(u^{\gamma}_{ij})=u^{\gamma}_{ij}$$ for all $\gamma\in\widehat{\mathcal{G}}$ such that ${\rm mul}(B,\gamma)>0$ and $1\leq i,j\leq d_{\gamma}$. Note that $\kappa^2$ is a linear multiplicative map on $\mathscr{A}$. Hence $\kappa^2$ is the identity map when restricted on the algebra $\mathscr{A}_2'$  generated by $u^{\gamma}_{ij}$'s  for all $\gamma\in\widehat{\mathcal{G}}$ such that ${\rm mul}(B,\gamma)>0$ and $1\leq i,j\leq d_{\gamma}$. If ${\rm mul}(B,\gamma)>0$, then ${\rm mul}(B,\gamma^c)>0$. Note that $\overline{u^{\gamma}}=(u_{ij}^{\gamma*})_{1\leq i,j\leq d_{\gamma}}\in \gamma^c$ for all $\gamma\in\widehat{\mathcal{G}}$. So  $u_{ij}^{\gamma*}\in \mathscr{A}_2'$  for all $\gamma\in\widehat{\mathcal{G}}$ such that ${\rm mul}(B,\gamma)>0$ and $1\leq i,j\leq d_{\gamma}$, and $\mathscr{A}_2'$ is a $*$-algebra. Thus $\mathscr{A}_2'=\mathscr{A}_2$  where $\mathscr{A}_2$ is defined in Proposition~\ref{faithfulness} and is the $*$-algebra generated by $u^{\gamma}_{ij}$ for all $\gamma\in\widehat{\mathcal{G}}$ such that ${\rm mul}(B,\gamma)>0$ and $1\leq i,j\leq d_{\gamma}$.

Note that $\alpha$ is faithful,  hence $\mathscr{A}_2=\mathscr{A}$ by Proposition~\ref{faithfulness}. So $\kappa^2=id$ on $\mathscr{A}$. This completes the proof.
\end{proof}
\begin{remark}\label{rkac}
Theorem~\ref{inv state is tracial haar measure is tracial} includes Theorem 2.10 (i) in ~\cite{BS2012} as special cases.

However, the converse of Theorem~\ref{inv state is tracial haar measure is tracial} is not true.

By~\cite[Theorem 5.1]{Wang1999}, there exists an ergodic and faithful action $\alpha$ of $A_u(n)$ on the Cuntz algebra $\mathcal{O}_n$ by $$\alpha(S_j)=\sum_{i=1}^n S_i\otimes u_{ij}.$$ Although $A_u(n)$ is a Kac algebra, there is no tracial state on $\mathcal{O}_n$.
\end{remark}
\section{Actions on compact Hausdorff spaces}

In this section, we suppose that a compact quantum group $\A$ acts on a compact Hausdorff space $X$  by $\alpha$ and denote $C(X)$ by $B$. Let $ev_x$ be the evaluation functional on $B$ at $x\in X$, i.e., $ev_x(f)=f(x)$ for all $f\in B$.

First, we derive some basic properties of invariant subsets of $X$ and invariant states on $B$. Next we show that the existence of minimal invariant subsets of $X$ and formulate the concept of compact quantum group orbits and use it as a tool to study ergodic compact quantum group actions on compact Hausdorff spaces.

\subsection{Invariant subsets and invariant states}

For a closed subset $Y$ of $X$, let $J_Y=\{f\in B|\, f=0\,{\rm on}\, Y\}$ and $\pi_Y$ be the quotient map from $B$ onto $B/J_Y$. Consider that a compact quantum group $\A$ acts on $X$ by $\alpha$. We say that $Y$ is an $\alpha$-\textbf{invariant subset} of $X$ if  $J_Y$ is an $\alpha$-invariant ideal of $B$.

Define the induced action $\alpha_Y$ of $\A$ on $Y$ by $\alpha_Y(f+J_Y)=(\pi_Y\otimes id)\alpha(f)$ for $f\in B$. For a state $\mu$ on $B$, since $B$ is commutative, $N_{\mu}=\{f\in B|\mu(f^*f)=0\}$ is a two-sided ideal of $B$. Let $X_{\alpha}=\{x\in X|f(x)=0\,\, \text{for all}\, f\in\ker{\alpha}\}$.

We now give another characterization of invariant subsets. First we need the following lemma.

\begin{lemma}\label{lemma for a characterization of invariant subsets}
For a closed subset $Y$ of $X$ and $f\in B$, $(\pi_Y\otimes id)\alpha(f)=0$ if and only if $(ev_x\otimes id)\alpha(f)=\alpha(f)(x)=0$ for all $x$ in $Y$.
\end{lemma}

\begin{proof}
Suppose that $(\pi_Y\otimes id)\alpha(f)=0$. For any $x$ in $Y$, we define a linear functional $\widetilde{ev_x}$ on $B/J_Y$ by $\widetilde{ev_x}(f+J_Y)=f(x)$ for all $f\in B$. If $f\in J_Y$, then $f(x)=0$ for all $x$ in $Y$. Hence  $\widetilde{ev_x}$ is well-defined. Furthermore, $\widetilde{ev_x}\pi_Y=ev_x$. Applying $\widetilde{ev_x}\otimes id$ to both sides of $(\pi_Y\otimes id)\alpha(f)=0$, we get $(ev_x\otimes id)\alpha(f)=0$ for all $x$ in $Y$.

On the other hand, for all $x$ in $Y$ and some $f\in B$, if $(ev_x\otimes id)\alpha(f)=0$, then $(\widetilde{ev_x}\pi_Y\otimes id)\alpha(f)=0$. Note that $(\pi_Y\otimes id)\alpha(f)\in (B/J_Y)\otimes A\cong C(Y)\otimes A\cong C(Y,A)$. Hence for all $x\in Y$, if $(\widetilde{ev_x}\otimes id)(\pi_Y\otimes id)\alpha(f)=0$, then $(\pi_Y\otimes id)\alpha(f)=0$.
\end{proof}

Using Lemma~\ref{lemma for a characterization of invariant subsets}, we have the following.

\begin{proposition}\label{characterization2 of invariant subset}
A closed subset $Y$ of $X$ is $\alpha$-invariant if and only if $(ev_x\otimes id)\alpha(f)=0$ for all $x$ in $Y$ and $f$ in $J_Y$.
\end{proposition}

Next we obtain some properties of $X_{\alpha}$.

\begin{proposition}\label{Xalpha is invariant and infinite}
The following hold:
\begin{enumerate}
\item The closed subset $X_{\alpha}$  of $X$ is $\alpha$-invariant.
\item If $X$ contains infinitely many points, then $X_{\alpha}$ also contains infinitely many points.
\end{enumerate}
\end{proposition}
\begin{proof}
(1)The invariance of $X_{\alpha}$ follows from Proposition~\ref{characterization2 of invariant subset}.

(2)Suppose that $X_{\alpha}$ has finitely many points. Thus $\alpha(B)\cong B/\ker{\alpha}\cong C(X_{\alpha})$ is finite dimensional.  Let $\varepsilon$ be the counit of $\mathcal{G}$ and  $\mathscr{B}$ be the Podle\'{s} subalgebra of $B$. Then $\alpha(\mathscr{B})\subseteq\alpha(B)$ is also finite dimensional. Since $(id\otimes\varepsilon)\alpha|_{\mathscr{B}}=id_{\mathscr{B}}$, we have that $\alpha$ is injective on $\mathscr{B}$. Hence $\mathscr{B}$ is  finite dimensional. This is a contradiction to that $B$ is infinite dimensional and that $\mathscr{B}$ is dense in $B$.
\end{proof}

Denote the quotient space of $X$ corresponding to $B^{\alpha}$ by $Y_{\alpha}$. It follows that $Y^{\alpha}=X/\sim$, and for $x,y\in X$, one have that $x\sim y$ if and only if $(ev_x\otimes h)\alpha=(ev_y\otimes h)\alpha$.

For  an $\alpha$-invariant subset $Y$ of $X$, let $\pi_Y: B\to B/J_Y$ be the quotient map. Recall that the induced action $\alpha_Y$ on $B/J_Y$ is defined by $\alpha_Y(f+J_Y)=(\pi_Y\otimes id)\alpha(f+J_Y)$ for all $f\in B$.

\begin{lemma}\label{invariance of pull back}
Let $Y$  be an $\alpha$-invariant subset of $X$. For an $\alpha_Y$-invariant state $\mu$ on $B/J_Y$, the pullback $\mu_Y:=\mu\circ\pi_Y$ is an $\alpha$-invariant state on $B$ and ${\rm supp}\mu_Y={\rm supp}\mu$.
\end{lemma}
\begin{proof}
The first assertion follows from the definition of invariant states.

For any  $f$ in $B$ , we have that $\mu_Y(f^*f)=0$ if and only if $\mu((f+J_Y)^*(f+J_Y))=0$. This says that the support of $\mu_Y$ is the same as $\mu$.
\end{proof}

\begin{proposition}\label{support of invariant measure is invariant}
Suppose that $\mu$ is an $\alpha$-invariant state on $B$. The support of $\mu$, denoted by ${\rm supp}\mu$, is an $\alpha$-invariant subset of $X$.
\end{proposition}

\begin{proof}
This follows from Theorem~\ref{invariant support}. See Example~\ref{ex:Inv sub}(2) for explanations.
\end{proof}

\subsection{Minimal invariant subsets and orbits}

In this section, we define compact quantum group orbits, derive some basic properties of it, and describe minimal invariant subsets. At last, we show that under  actions of coamenable compact quantum groups, orbits are minimal invariant subsets.

\begin{definition}\label{orbit}
Let $\mathcal{G}$ act on $X$ by $\alpha$. For $x\in X$, denote the  $*$-homomorphism $(ev_x\otimes id)\alpha: B \to A$  by $\alpha_x$, and let
$$\mathcal{M}_x=\{y\in X|f(y)=0\, \,{\rm for\, all}\,\, f\in \ker{\alpha_x}\}.$$  We call the subset
$$\{x'\in X|(ev_x\otimes h)\alpha=(ev_{x'}\otimes h)\alpha\}$$ of $X$
the \textbf{orbit} of  $x$, and denote it by ${\rm Orb}_x$.
\end{definition}

Next we show that $\mathcal{M}_x$ is an invariant subset.

\begin{lemma}\label{Mx is nonepmty and closed}
For any $x\in X$, the set $\mathcal{M}_x$ is a nonempty closed subset of $X$.
\end{lemma}
\begin{proof}
Note that $\alpha_x\neq 0$, thus $\mathcal{M}_x$ is a nonempty closed subset of $X$.
\end{proof}

For every $x\in X$, use $A_x$ to denote the C*-subalgebra $\alpha_x(B)$ of $A$. Now $A_x$ is isomorphic to $B/\ker{\alpha_x}$. We denote this isomorphism by $\widetilde{\alpha_x}:B/\ker{\alpha_x}\to A_x$ and the quotient map from $B$ onto $B/\ker{\alpha_x}$ by $\pi_x$. We get
\begin{equation} \label{eq:alphax}
\widetilde{\alpha_x}\pi_x=\alpha_x.
\end{equation}

A different form of the following lemma is proven in~\cite[Thoerem 3.2]{Goswami2010} by Goswami. We write down a proof here in detail based on his argument for completeness.

\begin{lemma} \label{prelimilary lemma for minimality of mx}
Suppose that $\mathcal{G}$ acts on a compact Hausdorff space $X$ by $\alpha$. For every $x\in X$, the coproduct $\Delta$ maps $A_x$ into $A_x\otimes A$. Moreover, we have that
\begin{equation} \label{eq:identity for minimality of mx}
(\pi_x\otimes id)\alpha=(\widetilde{\alpha_x}^{-1}\otimes id)\Delta\alpha_x.
\end{equation}
\end{lemma}

\begin{proof}
We first show that $\Delta$ maps $A_x$ into $A_x\otimes A$. For every $f\in B$, we have that
\begin{align*}
\Delta((ev_x\otimes id)\alpha(f))&=(ev_x\otimes id\otimes id)(id\otimes\Delta)\alpha(f)\\
                                 &=(ev_x\otimes id\otimes id)(\alpha\otimes id)\alpha(f)\\
                                 &=(\alpha_x\otimes id)\alpha(f).
\end{align*}
Therefore, we obtain that $\Delta(A_x)\subseteq A_x\otimes A$ which guarantees that the right hand side of equation~\eqref{eq:identity for minimality of mx} is well defined.

Secondly, for every $f\in B$, we get
\begin{align*}
&(\widetilde{\alpha_x}^{-1}\otimes id)\Delta\alpha_x(f)=(\widetilde{\alpha_x}^{-1}\otimes id)(\alpha_x\otimes id)\alpha(f)=(\pi_x\otimes id)\alpha(f).
\end{align*}
\end{proof}

\begin{proposition}\label{mx is invariant}
For every $x\in X$, the set $\mathcal{M}_x$ is an $\alpha$-invariant subset of $X$.
\end{proposition}
\begin{proof}
It suffices to show that $(\pi_x\otimes id)\alpha(f)=0$ for every $f\in \ker\alpha_x$, which follows from equation~\eqref{eq:identity for minimality of mx}.
\end{proof}

\begin{proposition} \label{minimal invariant space is mx}
For an $\alpha$-invariant subset $Y$ of $X$, $\mathcal{M}_x\subseteq Y$ for every $x\in Y$. If $Y$ is a minimal $\alpha$-invariant subset of $X$, then $Y=\mathcal{M}_x$ for every $x\in Y$.
\end{proposition}
\begin{proof}
Since $Y$ is $\alpha$-invariant, by Proposition~\ref{characterization2 of invariant subset}, we have that $(ev_x\otimes id)\alpha(f)=0$ for all $f\in J_Y$ and $x\in Y$, which is to say, if $f|_Y=0$, then $f|_{\mathcal{M}_x}=0$ for every $x\in Y$.
Hence $\mathcal{M}_x\subseteq Y$ for every $x\in Y$.

By Proposition~\ref{mx is invariant}, the set $\mathcal{M}_x$ is $\alpha$-invariant. If $Y$ is a minimal $\alpha$-invariant subset of $X$, then $\mathcal{M}_x=Y$ for all $x\in Y$.
\end{proof}

Recall that $X_{\alpha}=\{x\in X|f(x)=0\,\,{\rm for\,all} \,f\in \ker{\alpha}\}$.

\begin{proposition}\label{a characterization of Xalpha}
$X_{\alpha}=\overline{\bigcup_{x\in X}\mathcal{M}_x}$.
\end{proposition}
\begin{proof}
Note that $\alpha(f)=0$ if and only if $(ev_x\otimes id)\alpha(f)=0$ for all $x\in X$. That is, $\alpha(f)=0$ if and only if $f|_{\mathcal{M}_x}=0$ for all $x\in X$. This is equivalent to say that $X_{\alpha}=\overline{\bigcup_{x\in X}\mathcal{M}_x}$.
\end{proof}

\begin{theorem}\label{a characterization of injectivity of alpha}
The action $\alpha$ of $\mathcal{G}$ on $X$ is injective if and only if $X=X_{\alpha}=\overline{\bigcup_{x\in X}\mathcal{M}_x}$.
\end{theorem}
\begin{proof}
It is easy to check that $\alpha$ is injective if and only if $X=X_{\alpha}$. By Proposition~\ref{a characterization of Xalpha}, we finish the proof.
\end{proof}

Next, we want to show that every orbit is an invariant subset.

Recall that $B^{\alpha}\cong C(Y_{\alpha})$ and we denote the canonical quotient map from $X$ onto $Y_{\alpha}$ by $\pi$. Then we have the following,
\begin{lemma}\label{orbit is fiber}
For every $y\in Y_{\alpha}$, two points $x_1$ and $x_2$ are in $\pi^{-1}(y)$  if and only if $x_1$ and $x_2$ are in the same orbit.
\end{lemma}
\begin{proof}
Note that $B^{\alpha}=(id\otimes h)\alpha(B)$. We have that $x_1, x_2 \in \pi^{-1}(y)$ for $y\in Y_{\alpha}$ if and only if $$(ev_{x_1}\otimes h)\alpha(g)=(ev_{y}\otimes h)\alpha(g)=(ev_{x_2}\otimes h)\alpha(g)$$ for every $g\in B$. That is to say, $x_1$ and $x_2$ are in the same orbit.
\end{proof}
\begin{theorem}\label{invariance of orbit}
For every $x\in X$, the orbit ${\rm Orb}_x$ is an $\alpha$-invariant subset of $X$.
\end{theorem}
\begin{proof}
By Proposition~\ref{characterization2 of invariant subset}, it suffices to show that for any $f\in C(X)$, if $f|_{{\rm Orb}_x}=0$, then
$(ev_{x'}\otimes id)\alpha(f)=0$ for every $x'\in {\rm Orb}_x$.

By Lemma~\ref{orbit is fiber}, there exists $y\in Y_{\alpha}$ such that $\pi^{-1}(y)={\rm Orb}_x$.

Let $f\in B$ such that $f|_{{\rm Orb}_x}=0$. For arbitrary $\varepsilon >0$, denote the closed subset $\{x\in X||f(x)|\geq \varepsilon\}$ by $E_{\varepsilon}$. Both $X$ and $Y_{\alpha}$ are compact Hausdorff spaces,

hence $\pi(E_{\varepsilon})$, denoted by $K_{\varepsilon}$, is also compact and Hausdorff. Since $y\notin K_{\varepsilon}$, by Urysohn's Lemma, there exists a $g_{\varepsilon}\in B^{\alpha}$, such that $0\leq g_{\varepsilon}\leq 1$, $g_{\varepsilon}(y)=0$ and  $g_{\varepsilon}|_{K_{\varepsilon}}=1$. Since $B^{\alpha}$ is a C*-subalgebra of $B$, the function $g_{\varepsilon}$ is also in $B$ and satisfies that $0\leq g_{\varepsilon}\leq 1$, $g_{\varepsilon}|_{{\rm Orb}_x}=0$ and $g_{\varepsilon}|_{E_{\varepsilon}}=1$.

Now consider $f-fg_{\varepsilon}$. Then $|f(x)- g_{\varepsilon}(x)f(x)|=0$ for every $x$ in $E_{\epsilon}$, and $|f(x)- g_{\varepsilon}(x)f(x)|<\varepsilon$ for all $x\in X\setminus E_{\varepsilon}$ since $|f(x)|<\varepsilon$ and $0\leq g_{\varepsilon}\leq 1$. Therefore $||f-f g_{\varepsilon}||<\varepsilon$ which implies
$$||(ev_{x'}\otimes id)\alpha(f)-(ev_{x'}\otimes id)\alpha(fg_{\varepsilon})||<\varepsilon$$ for  every
 $x'\in X$.

 Note that $g_{\varepsilon}\in B^{\alpha}$ and $g_{\varepsilon}|_{{\rm Orb}_x}=0$. For  every
 $x'\in {\rm Orb}_x$, we have that
 $$(ev_{x'}\otimes id)\alpha(fg_{\varepsilon})=(ev_{x'}\otimes id)(\alpha(f)(g_{\varepsilon}\otimes 1))=(ev_{x'}\otimes id)\alpha(f)g_{\varepsilon}(x')=0.$$
  Consequently, $||(ev_{x'}\otimes id)\alpha(f)||<\varepsilon$ for all $x'\in {\rm Orb}_x$. Note that $\varepsilon$ is arbitrary.
 So $(ev_{x'}\otimes id)\alpha(f)=0$ for every $x'\in {\rm Orb}_x$. This ends the proof.
\end{proof}
\begin{theorem}\label{a characterization of ergodic action by orbit}
The following are equivalent:
\begin{enumerate}
\item The action $\alpha$ is ergodic.
\item ${\rm Orb}_x=X$ for every $x\in X$.
\item There exists $x_0\in X$ such that ${\rm Orb}_{x_0}=X$.
\end{enumerate}
\end{theorem}
\begin{proof}
Obviously (2) implies (3). So we just prove that  (1) implies (2) and (3) implies (1).

(1)$\Rightarrow $(2). Suppose that $\alpha$ is ergodic. Then $(id\otimes h)\alpha(f)$ is a constant function on $X$ for every $f\in B$. Therefore,  $(ev_x\otimes h)\alpha(f)=(ev_{x'}\otimes h)\alpha(f)$ for all $x$ and $x'$ in $X$. Consequently ${\rm Orb}_x=X$.

(3)$\Rightarrow $(1). If there exists $x_0\in X$ such that ${\rm Orb}_{x_0}=X$. We have that $(ev_x\otimes h)\alpha(f)=(ev_{x_0}\otimes h)\alpha(f)$ for every $f\in B$ and  $x \in X$. This is equivalent to say that $(id\otimes h)\alpha(f)$ is a constant function on $X$ for every $f\in B$. Therefore $\alpha$ is ergodic.
\end{proof}

Since for an action $\alpha$ of $\mathcal{G}$ on $X$, every orbit ${\rm Orb}_x$ is invariant,  $\alpha$ induces an action $\delta_x$ of $\mathcal{G}$ on ${\rm Orb}_x$. Let $J_x=\{f\in B|\,f|_{{\rm Orb}_x}=0\}$ be the ideal consisting of continuous functions varnishing in ${\rm Orb}_x$ and $\pi_x$ be the quotient map from $B$ onto $B/J_x$. Then $\delta_x(f+J_x)=(\pi_x\otimes id)\alpha(f)$ for every $f$ in $B$.

\begin{lemma}\label{induced action on orbit is ergodic}
The action $\delta_x$ on ${\rm Orb}_x$ is ergodic.
\end{lemma}
\begin{proof}

 Let $\widetilde{ev}_x$ and $\widetilde{ev}_y$ be the evaluation functional on $B/J_x$ at any two points $x$ and $y$ in ${\rm Orb}_x$. Then for every $f$ in $B$
\begin{align*}
(\widetilde{ev}_y\otimes h)\delta_x(f+J_x)&=(\widetilde{ev}_y\otimes h)(\pi_x\otimes id)\alpha(f)\\
                                            &=(ev_y\otimes h)\alpha(f)=(ev_x\otimes h)\alpha(f)\\
                                            &=(\widetilde{ev}_x\otimes h)\delta_x(f+J_x).
\end{align*}
\end{proof}

\begin{proposition}\label{Inv subset's orbit}
Suppose that $Y$ is an $\alpha$-invariant subset of $X$. For every $y\in Y$, the orbit of $y$ under $\alpha_Y$ is
$Y\bigcap {\rm Orb}_y.$
\end{proposition}
\begin{proof}
For $y$ and $y'$ in $Y$, $(ev_y\otimes h)\alpha_Y=(ev_{y'}\otimes h)\alpha_Y$ if and only if $(ev_y\otimes h)\alpha=(ev_{y'}\otimes h)\alpha$.
\end{proof}

Let $\mu_x=(ev_x\otimes h)\alpha$ for $x\in X$. Then $\mu_x$ is an $\alpha$-invariant state on $B$. The support of $\mu_x$, denoted by ${\rm supp}\mu_x$, is an $\alpha$-invariant subset by Theorem~\ref{support of invariant measure is invariant}. Next we show that minimal invariant subsets are always of the form ${\rm supp}\mu_x$.

\begin{theorem}\label{support is minimal}
Consider a compact quantum group action $\alpha$ of $\A$ on a compact Hausdorff space $X$. For every $x\in X$, we have that ${\rm supp}\mu_x\subseteq\mathcal{M}_x\subseteq{\rm Orb}_x$ and that ${\rm supp}\mu_x$ is a minimal $\alpha$-invariant subset of $X$.  Moreover, if the Haar measure $h$ of $\mathcal{G}$ is faithful, then ${\rm supp}\mu_x=\mathcal{M}_x$.
\end{theorem}
\begin{proof}
 For every $x\in X$, if $f|_{\mathcal{M}_x}=0$ for some $f\in B$, then $\alpha_x(f^*f)=0$. It follows that $\mu_x(f^*f))=h(\alpha_x(f^*f))=0$, which means that $f|_{{\rm supp}\mu_x}=0$. Both $\mathcal{M}_x$ and ${\rm supp}\mu_x$ are closed subsets of $X$. Therefore, ${\rm supp}\mu_x\subseteq\mathcal{M}_x$.

By Proposition~\ref{minimal invariant space is mx}, for every $x\in X$, $\mathcal{M}_x\subseteq {\rm Orb}_x$ since $x\in {\rm Orb}_x$.

 Suppose that ${\rm supp}\mu_x$ is not minimal. Then there exists an $\alpha$-invariant subset $Y\subsetneq {\rm supp}\mu_x$. Denote the induced actions of $\mathcal{G}$ on $Y$ and ${\rm Orb}_x$ by $\alpha_Y$ and $\alpha_x$ respectively. Choose an $\alpha_Y$-invariant  state $\omega$ on $C(Y)$, and denote the pull back of it to $C({\rm Orb}_x)$ by $\omega_x$. It follows that $\omega_x$ is also an $\alpha_x$-invariant state on $C({\rm Orb}_x)$, and $\omega_x$ is different from $\mu_x$ since the support of $\omega_x$ which is contained in $Y$ by Lemma~\ref{invariance of pull back}, is a proper subset of the support of $\mu_x$. Since the action  $\alpha_x$ is ergodic by Lemma~\ref{induced action on orbit is ergodic}, this leads to a contradiction to the uniqueness of $\alpha_x$-invariant states on ${\rm Orb}_x$. So ${\rm supp}\mu_x$ is minimal.

 Now assume that $h$ is faithful. For $x\in X$, if $h(\alpha_x(f^*f))=\mu_x(f^*f)=0$ for $f\in B$, then $\alpha_x(f^*f)=0$ by the faithfulness of $h$. This implies that $\mathcal{M}_x\subseteq {\rm supp}\mu_x$.
\end{proof}

So far, all examples of compact quantum group actions on compact Hausdorff spaces illustrate that ${\rm supp}\mu_x=\mathcal{M}_x={\rm Orb}_x$. Later we show that every orbit is minimal if the compact quantum group is coamenable or the space $X$ is countable~(see Theorem~\ref{coamenable implies minimal orbits}, Corollary~\ref{finite is minimal} and Corollary~\ref{acc2}). Hence we give the following conjecture.
\begin{conjecture}\label{orbit minimal}
Consider a compact quantum quantum group action on $X$. Then ${\rm supp}\mu_x=\mathcal{M}_x={\rm Orb}_x$ for every $x\in X$.
\end{conjecture}

If we consider the reduced action $\alpha_r$ of $\mathcal{G}_r$ on $X$, then it turns out that two points are in the same orbit under $\alpha$ if and only if these two points are in the same orbit under $\alpha_r$.

\begin{proposition}\label{orbit unchange under reduced action}
For $x\in X$, the orbit of $x$ under the action of $\mathcal{G}$ on $X$ is the same as the orbit of $x$ under the action of $\mathcal{G}_r$ on $X$.
\end{proposition}

\begin{proof}
Two points $x$ and $y$ in $X$ are in the same orbit under the action $\alpha$ of $\mathcal{G}$ on $X$ if and only if
$(ev_x\otimes h)\alpha=(ev_y\otimes h)\alpha$. Also $x$ and $y$ are in the same orbit under the action $\alpha_r$ of $\mathcal{G}_r$ on $X$ if and only if $(ev_x\otimes h_r)\alpha_r=(ev_y\otimes h_r)\alpha_r$. Since $h_r\pi_r=h$, we have that
$$(ev_x\otimes h_r)\alpha_r=(ev_x\otimes h_r)(id\otimes\pi_r)\alpha=(ev_x\otimes h)\alpha.$$  Thus $(ev_x\otimes h_r)\alpha_r=(ev_y\otimes h_r)\alpha_r$ if and only if $(ev_x\otimes h)\alpha=(ev_y\otimes h)\alpha$. This completes the proof.
\end{proof}

A compact quantum group $\mathcal{G}$ is called \textbf{coamenable} if  its Haar measure $h$ is faithful and its counit $\varepsilon$ is bounded.

\begin{theorem}\label{coamenable implies minimal orbits}
 Suppose that a coamenable compact quantum group $\mathcal{G}$ acts on $X$ by $\alpha$. For any $\alpha$-invariant subset $Y$ of $X$,  we have that $$Y=\bigcup_{x\in Y}{\rm Orb}_x.$$
 Consequently, ${\rm Orb}_x$ is a minimal $\alpha$-invariant subset of $X$ for every $x\in X$.
\end{theorem}

\begin{proof}
The second assertion follows directly from the first.
The first assertion is equivalent to  that if $Y$ is an $\alpha$-invariant subset of $X$ and $x$ is in $Y$, then  ${\rm Orb}_x\subseteq Y$. 

Suppose that  $f$ is a positive  continuous function  on $X$ and $f|_{Y}=0$. Since $Y$ is $\alpha$-invariant, we have that $(ev_x\otimes id)\alpha(f)=0$ for all $x\in Y$. For every $y$ in ${\rm Orb}_x$, since
$(ev_x\otimes h)\alpha=(ev_y\otimes h)\alpha$, we obtain that $0=(ev_x\otimes h)\alpha(f)=(ev_y\otimes h)\alpha(f)$. Since $h$ is faithful and $f$ is positive, we have that $(ev_y\otimes id)\alpha(f)=0$. The counit $\varepsilon$ is bounded, hence $(id\otimes\varepsilon)\alpha(g)=g$ for all $g\in B$. So
$$0=\varepsilon((ev_y\otimes id)\alpha(f))=ev_y((id\otimes\varepsilon)\alpha(f))=f(y).$$ Thus we get that if $f\geq 0$ and $f|_Y=0$, then $f(y)=0$ for any $y\in {\rm Orb}_x$. Note that both  ${\rm Orb}_x$ and $Y$ are closed subsets of $X$. Therefore ${\rm Orb}_x\subseteq Y$ for any $x\in Y$.
\end{proof}

\subsection{Actions on finite spaces}

Throughout this subsection, we consider  a compact quantum group action $\alpha$ on a finite space $X_n=\{x_1,x_2,...,x_n\}$ with $n$ points. Let $a_{ij}=(ev_i\otimes id)\alpha(e_j)$  where $e_j$ is the characteristic function of $\{x_j\}$ and $ev_i$ is the evaluation functional on $B=C(X_n)$ at the point $x_i$ for $1\leq i,j\leq n$.

The main result in this section is the following characterization of ergodic actions on finite spaces.
\begin{theorem}\label{erdodic action on finite spaces}
Consider a compact quantum group action $\alpha$ of $\mathcal{G}$ on $B$. The following are equivalent.
\begin{enumerate}
\item The action $\alpha$ is ergodic.
\item $h(a_{ij})=\frac{1}{n}$ for the Haar measure $h$ of $\mathcal{G}$ and $1\leq i,j\leq n$.
\item All $a_{ij}$'s are nonzero for $1\leq i,j\leq n$.
\item Every $\alpha$-invariant state $\psi$ of $B$ satisfies $\psi(e_i)=\frac{1}{n}$ for all $1\leq i\leq n$.
\end{enumerate}
\end{theorem}
\begin{proof}
We prove this theorem by showing that (1)$\Rightarrow$ (2)$\Rightarrow$(3)$\Rightarrow$(1) and (2)$\Leftrightarrow$(4).

(1)$\Rightarrow$ (2). If $\alpha$ is ergodic, then by Proposition~\ref{a characterization of ergodic action by orbit}, any two points $x_i$ and $x_j$ in $X_n$ are in the same orbit. That is to say $(ev_i\otimes h)\alpha=(ev_j\otimes h)\alpha$. Hence
$(ev_i\otimes h)\alpha(e_k)=(ev_j\otimes h)\alpha(e_k)$ for $1\leq i,j,k\leq n$. Therefore
\begin{equation}\label{eqh}
h(a_{ik})=h(a_{jk})
\end{equation}
for $1\leq i,j,k\leq n$.

 We first show  that $h(a_{ij})$ is nonzero for all $1\leq i,j\leq n$. Suppose not, then there exist $1\leq i,k\leq n$ such that $h(a_{ik})=0$.
 By equation~\eqref{eqh}, we have that $h(a_{jk})=0$ for all $1\leq j\leq n$. Let $\mu$ be the unique invariant state on $B$.

 Note that $B$ is finite dimensional. So $\mathscr{B}=B$ and $\alpha$ is injective on $B$. Hence $a_{jk}=(ev_j\otimes id)\alpha(e_k)\in \mathscr{A}$ for all $1\leq j\leq n$. Since $h$ is faithful on $\mathscr{A}$ and $a_{jk}^*=a_{jk}=a_{jk}^2$, we have that $a_{jk}=0$ for  all $1\leq j\leq n$. Thus $\alpha(e_k)(j)=(ev_j\otimes id)\alpha(e_k)=a_{jk}=0$ for all $1\leq j\leq n$. Therefore $\alpha(e_k)=0$. This is a contradiction to the injectivity of $\alpha$.

Now $h(a_{ij})>0$ for all $1\leq i,j\leq n$. 
If we can prove  $h(a_{ij})=h(a_{ik})$ for all $1\leq i,j,k\leq n$, then combining this  with  $\sum_{j=1}^n a_{ij}=1$, we get that $h(a_{ij})=\frac{1}{n}$ for $1\leq i,j\leq n$.

By the invariance of the Haar measure we have that $(h\otimes\omega)\Delta=h$ for every state $\omega$ on $A$. Apply this to $a_{ij}$. Since $\Delta(a_{ij})=\sum_{k=1}^n a_{ik}\otimes a_{kj}$, we have that
$$\sum_{k=1}^n h(a_{ik})\omega(a_{kj})=h(a_{ij}).$$
Since $h(a_{ij})>0$ for all $1\leq i,j\leq n$, every $a_{ij}$ is a nonzero projection in $A$. For any $1\leq l,j\leq n$, there exists a state $\omega_{lj}$ on $A$ such that $\omega_{lj}(a_{lj})=1$. It follows that
$$h(a_{ij})=\sum_{k=1}^n h(a_{ik})\omega_{lj}(a_{kj})=h(a_{il})+\sum_{k\neq l}h(a_{ik})\omega_{lj}(a_{kj}).$$
Hence for all $1\leq i,j,l\leq n$, $h(a_{ij})\geq h(a_{il})$, and symmetrically, $h(a_{il})\geq h(a_{ij})$. So $h(a_{ij})=h(a_{ik})$ for any $1\leq i,j,k\leq n$.

(2)$\Rightarrow$(3). This is trivial.

(3)$\Rightarrow$(1). Suppose that $\alpha(f)=f\otimes 1$ for some $f=\sum_{i=1}^n f_ie_i$ with $f_i\in\Complex$. Since
$\alpha(f)=\sum_{i=1}^n \sum_{k=1}^n f_i e_k\otimes a_{ki}$ and $f\otimes 1=\sum_{k=1}^n f_ke_k\otimes 1$, we get
$\sum_{i=1}^n f_ia_{ki}=f_k$ for every $1\leq k\leq n$. Since $a_{ki}$ is a nonzero projection in $A$ for all $1\leq k,l\leq n$, there exists a state $\omega_{kl}$ on $A$ such that $\omega_{kl}(a_{kl})=1$. Note that $\sum_{j=1}^n a_{kj}=1$. So $\omega_{kl}(a_{kj})=\delta_{lj}$ for all $1\leq j\leq n$. Applying $\omega_{kl}$ to both sides of $\sum_{i=1}^n f_ia_{ki}=f_k$, we get $f_l=f_k$ for all $1\leq k,l\leq n$. Therefore $f$ is a constant function on $X_n$ and the action $\alpha$ is ergodic.

(2)$\Rightarrow$(4). Every invariant measure $\psi$ can be written as $\psi=(\phi\otimes h)\alpha$ for some state $\phi$ on $B$. Hence
$$\psi(e_i)=(\phi\otimes h)\alpha(e_i)=\sum_{k=1}^n \phi(e_k)h(a_{ki})=\frac{1}{n}\sum_{k=1}^n \phi(e_k)=\frac{1}{n}.$$

(4)$\Rightarrow$(2). For every state $\phi$ on $B$, the state $(\phi\otimes h)\alpha$ is invariant. Especially choose $\phi=ev_j$. Then
$$\frac{1}{n}=(ev_j\otimes h)\alpha(e_i)=h((ev_j\otimes id)\alpha(e_i))=h(a_{ji}).$$
\end{proof}

\begin{proposition}\label{orbits in finite spaces}
Suppose that a compact quantum group $\mathcal{G}$ acts on $X_n$ by $\alpha$. Then two points $x_i$ and $x_j$ in $X_n$ are in the same orbit if and only if  $a_{ij}\neq 0$. If $a_{ij}\neq 0$, then $h(a_{ij})=\frac{1}{m_i}$ where $m_i$ is the cardinality of ${\rm Orb}_{x_i}$.

\end{proposition}
\begin{proof}
Suppose that $a_{ij}\neq 0$. If $x_j$ is not in ${\rm Orb}_{x_i}$, then $e_j|_{{\rm Orb}_{x_i}}=0$. Note that ${\rm Orb}_{x_i}$ is an $\alpha$-invariant subset of $X_n$. Thus $a_{ij}=(ev_i\otimes id)\alpha(e_j)=0$, which is a contradiction. This proves the sufficiency.

On the other hand, suppose that $x_j\in {\rm Orb}_{x_i}$. Let $J_i=\{f\in B|\,\,f|_{{\rm Orb}_{x_i}}=0\}$ be the ideal of continuous functions varnishing on ${\rm Orb}_{x_i}$ and $\pi_i$ be the quotient map from $B$ onto $B/J_i$. Denote the induced action on ${\rm Orb}_{x_i}$ by $\alpha_i$, which is defined by $\alpha_i(f+J_i)=(\pi_i\otimes id)\alpha(f)$ for $f\in B$. Let $\widetilde{ev_j}$ be the evaluation functional on $B/J_i$ at $x_j$. Denote $(\widetilde{ev_i}\otimes id)\alpha_i(e_j)$ by $\widetilde{a_{ij}}$. Then $\widetilde{a_{ij}}=a_{ij}$ by the definition of $\alpha_i$.
By Theorem~\ref{erdodic action on finite spaces}, the action $\alpha_i$ on ${\rm Orb}_{x_i}$ is ergodic. Hence $h(a_{ij})=h(\widetilde{a_{ij}})=\frac{1}{m_i}$ where $m_i$ is the cardinality of ${\rm Orb}_{x_i}$. This proves the necessity.
\end{proof}

\begin{corollary}\label{finite is minimal}
Suppose that a compact quantum group $\A$~acts on $X_n$ ergodically. Then ${\rm supp}\mu_x=\mathcal{M}_x={\rm Orb}_x=X_n$ for every $x\in X_n$, that is, every orbit ${\rm Orb}_x$ is a minimal invariant subset.
\end{corollary}
\begin{proof}
Note that ${\rm supp}\mu_x\subseteq\mathcal{M}_x\subseteq{\rm Orb}_x$ for every $x\in X$ by Theorem~\ref{support is minimal}. So it suffices to show that if a positive $f$ in $B$ satisfies that $\mu_x(f)=0$ for some $x\in X_n$, then $f=0$. Write $f$ as $\sum_{i=1}^n f_ie_i$ with $0\leq f_i\in\Complex$ and let $x=x_k$ for some $1\leq k\leq n$. Then by Theorem~\ref{erdodic action on finite spaces}, we have $$0=\mu_{x}(f)=(ev_k\otimes h)\alpha(f)=\sum_{i=1}^n f_i(ev_k\otimes h)\alpha(e_i)=\sum_{i=1}^n f_i h(a_{ki})=\frac{1}{n}\sum_{i=1}^n f_i.$$  Hence $f_i=0$ for all $1\leq i\leq n$ and $f=0$.
\end{proof}

\subsection{Non-atomic invariant measures}

Our main theorem in this subsection is the following.
\begin{theorem}\label{nonatomic inv measure}
If a compact quantum group $\A$ acts ergodically by $\alpha$ on a compact Hausdorff space $X$ with infinitely many points, then the unique $\alpha$-invariant measure $\mu$ of $X$ is non-atomic. That is, every point of $X$ has zero $\mu$-measure.
\end{theorem}

Denote $C(X)$ by $B$. For $y\in X$, denote by $e_y$  the characteristic function of  $\{y\}$.  For a compact quantum group  action $\alpha: B\rightarrow B\otimes A$, we use $ev_x$  to denote the evaluation functional on $B$ at a point $x\in X$.

Take a regular Borel probability measure $\mu$ on $X$.  Denote  $\mu(\{x\})$ by $\mu_x$ and  define a linear functional $\nu_x$ on $B$ by $\nu_x(f)=f(x)\mu_x$ for all $f\in B$. With abuse of notation, we also use $\mu$ to denote the corresponding linear functional on $B$ such that $\mu(f)=\int_{X}f\,d\mu$ for $f\in B$. For a subset $U$ of $X$, if an $f\in B$ satisfies that $0\leq f\leq 1$ and
$f|_U=1$, then we write it as
$U\prec f$. If $f$ satisfies that $0\leq f\leq 1$ and  ${\rm support of} \,f\subseteq U$, then we denote it by
$f\prec U$.

Before proceeding to the main theorem, we prove some preliminary lemmas first.
\begin{lemma}\label{natm1}
Suppose that a compact quantum group $\A$ acts on a compact Hausdorff space $X$ by $\alpha$. Take an $\alpha$-invariant measure $\mu$ on $X$. If for two points $x$ and $y$ in $X$,  we have that~$\mu_x>\mu_y$, then there exists an open neighborhood $V$ of $y$ satisfying that $(ev_x\otimes id)\alpha(g)=0$ for all $g\in B$ with $g\prec V$.
\end{lemma}
\begin{proof}
Note that $\mu$ is a state on $B$ and $X$ is a compact Haudorff space. Hence $\mu$ is a regular Borel measure on $X$ by the Riesz representation theorem. Since $\mu_x>\mu_y$, there exists an open neighborhood $U$ of $y$ such that $\mu_x>\mu(U)$.   We claim that
$$\|(ev_x\otimes id)\alpha(f)\|<1$$ for all $f\in B$ with $f\prec U$.
Since $0\leq f\leq 1$, we have that $\|(ev_x\otimes id)\alpha(f)\|\leq 1$. If $\|(ev_x\otimes id)\alpha(f)\|=1$, then there exists a state $\phi$ on $A$ such that $\phi((ev_x\otimes id)\alpha(f))=\|(ev_x\otimes id)\alpha(f)\|=1$ since $(ev_x\otimes id)\alpha(f)\geq 0$. Moreover,
\begin{align*}
(\mu\otimes \phi)\alpha(f)&=\phi((\mu\otimes id)\alpha(f))=\phi(\int_X\,(ev_x\otimes id)\alpha(f)\,d\mu)   \\ \notag
                          &\geq\phi((ev_x\otimes id)\alpha(f))\mu_x=\mu_x.
\end{align*}
Since $\mu$ is $\alpha$-invariant, on the other hand
$$(\mu\otimes \phi)\alpha(f)=\phi((\mu\otimes id)\alpha(f))=\phi(\mu(f)1_A)=\mu(f).$$  Therefore combining these, we get that $\mu(f)\geq \mu_x$. Since $f\prec U$, we also have that $\mu_x>\mu(U)\geq \mu(f)$. This leads to a contradiction.
Hence $\|(ev_x\otimes id)\alpha(f)\|<1$ for all $f\in B$ with $f\prec U$.

Since $X$ is a compact Hausdorff space, there exist an open subset $V$ and a compact subset $K$ of $X$ such that $y\in V\subseteq K\subseteq U$.

By  Urysohn's lemma, there is an $f\in B$ such that $K\prec f\prec U$. For any $g\in B$ with $g\prec V$,
we see that $0\leq g\leq f^n$ for every positive integer $n$.
Thus
$$\|(ev_x\otimes id)\alpha(g)\|\leq \|(ev_x\otimes id)\alpha(f^n)\|=\|(ev_x\otimes id)\alpha(f)\|^n\to 0$$ as $n\to\infty$.
Therefore $(ev_x\otimes id)\alpha(g)=0$.
\end{proof}

\begin{lemma}\label{natm2}
Suppose that a compact quantum group $\A$ has the faithful Haar measure and acts ergodically by $\alpha$ on a compact Hausdorff space $X$ with infinitely many points. Denote the unique $\alpha$-invariant measure on $X$ by $\mu$. Assume that there exists some $x\in X$ such that $\mu_x>0$.   
Let $E_1=\{y\in X|\mu_y= \max\{\mu_x|x\in X\}\}$. For any $f\in B$, if $f|_{E_1}=0$, we have $\alpha(f)=0$.
\end{lemma}
\begin{proof}
First $E_1$ is a finite subset of $X$ since  $\mu$ is a finite measure on $X$. Let $E_1=\{x_1,...,x_n\}$ and $ev_i=ev_{x_i}$ for $1\leq i\leq n$. For any $\epsilon>0$, there exists an open neighborhood $V_i$ of $x_i$ for each $x_i\in E_1$ such that $|f(x)|<\epsilon$ for all $x\in \bigcup_{i=1}^n V_i$. For any $y\notin E_1$, by Lemma~\ref{natm1}, there exists an open neighborhood $V_y$ of $y$ such that $V_y\bigcap E_1=\emptyset$ and $(ev_i\otimes id)\alpha(g)=0$ for all $g\in B$ with $g\prec V_y$ and  all $1\leq i\leq n$. Then $\mathcal{V}=\{V_y\}_{y\notin E_1}\bigcup\{V_i\}_{i=1}^n$ is an open cover of $X$. Since $X$ is a compact Hausdorff space, there exists a finite subcover $\mathcal{V}'$ of $\mathcal{V}$. Let $\{g_V\}_{V\in\mathcal{V}'}$ be a partition of unity of $X$ subordinate to $\mathcal{V}'$. Then $f=\sum_{V\in\mathcal{V}'}fg_V$.

Now let $i=1$ for convenience. By Lemma~\ref{natm1}, we have that $(ev_1\otimes id)\alpha(g_V)=0$ for all $V\in \mathcal{V}'\setminus\{V_i\}_{i=1}^n$. Hence
\begin{align*}
(ev_1\otimes id)\alpha(f)&=(ev_1\otimes id)\alpha(\sum_{V\in\mathcal{V}'}fg_V)     \\
&=\sum_{V\in\mathcal{V}'\bigcap\{V_i\}_{i=1}^n}(ev_1\otimes id)\alpha(fg_V)+\sum_{V\in\mathcal{V}'\setminus\{V_i\}_{i=1}^n}(ev_1\otimes id)\alpha(fg_V)  \\
&=\sum_{V\in\mathcal{V}'\bigcap\{V_i\}_{i=1}^n}(ev_1\otimes id)\alpha(fg_V).
\end{align*}

Take any $x\in X$. If $x\in \bigcup_{i=1}^n V_i$, then
$|\sum_{V\in\mathcal{V}'\bigcap\{V_i\}_{i=1}^n}f(x)g_V(x)|\leq |f(x)|<\epsilon$. If $x\notin \bigcup_{i=1}^n V_i$, then $\sum_{V\in\mathcal{V}'\bigcap\{V_i\}_{i=1}^n}f(x)g_V(x)=0$. Therefore $\|\sum_{V\in\mathcal{V}'\bigcap\{V_i\}_{i=1}^n}fg_V\|\leq \epsilon$.

Thus
$$\|(ev_1\otimes id)\alpha(f)\|=\|(ev_1\otimes id)\alpha(\sum_{V\in\mathcal{V}'\bigcap\{V_i\}_{i=1}^n}fg_V)\|\leq\epsilon.$$
Since $\epsilon$ is arbitrary, we have that $(ev_1\otimes id)\alpha(f)=0$.  Note that $(ev_1\otimes id)\alpha$ is a $*$-homomorphism, so $(ev_1\otimes id)\alpha(f^*f)=0$. The action $\alpha$ is ergodic, hence $(ev_x\otimes h)\alpha(f^*f)=(ev_1\otimes h)\alpha(f^*f)=0$ for any $x\in X$. The Haar measure $h$ is faithful and $(ev_x\otimes id)\alpha(f^*f)\geq 0$, therefore $(ev_x\otimes id)\alpha(f^*f)=0$ for all $x\in X$ which means $\alpha(f)=0$.
\end{proof}
Now we are ready to prove the main theorem in this subsection.
\begin{proof}[Proof of Theorem~\ref{nonatomic inv measure}]\
We can assume  the Haar measure $h$ of $\A$~is faithful otherwise we replace $\alpha$ by the reduced compact quantum group action $\alpha_r$ of $\mathcal{G}_r$ which has the faithful Haar measure. The action $\alpha_r$ is also ergodic by Corollary~\ref{ergodic iff reduced ergodic}. Moreover, a state on $B$ is $\alpha$-invariant if and only if it is $\alpha_r$-invariant~(see the argument preceding Corollary~\ref{ergodic iff reduced ergodic}).

Suppose that $\mu(\{x\})>0$ for some $x\in X$. Define $E_1=\{x_1,...,x_n\}$  as in Lemma~\ref{natm2}. Let $\mathscr{B}$ be the Podl\'{e}s algebra of $B=C(X)$. Define a linear map $T$ from $\alpha(\mathscr{B})$ into $\Complex^n$ by
$$T(\alpha(f))=(f(x_1),f(x_2),...,f(x_n))$$ for all $f\in\mathscr{B}$. Note that $\alpha$ is injective on $\mathscr{B}$. So $T$ is well-defined. Also $T$ is linear. By Lemma~\ref{natm2}, $T$ is injective. The space $X$ contains infinitely many points, hence $B$ is infinite dimensional. Since $\mathscr{B}$ is a dense subspace of $B$, we have that $\mathscr{B}$ is also infinite dimensional. This leads to a contradiction to that $\Complex^n$ is finite dimensional and that $T$ is injective.
\end{proof}

\subsection{Actions on countable compact Hausdorff spaces}

In this subsection, we consider compact quantum group actions on a compact Hausdorff space $X_\infty$ with countably infinitely many points. Within the section, the notation $B$ stands for $C(X_\infty)$. Denote by $X_I$ the set of isolated points and by $X_A$ the set of accumulation points of $X_\infty$.

The main theorem of this subsection follows directly from Theorem~\ref{nonatomic inv measure}.
\begin{theorem}\label{quantum homogenous space}
$X_\infty$ is not a quantum homogeneous space.
\end{theorem}
\begin{proof}
For every Borel probability measure $\mu$ on $X_\infty$, there exists an $x\in X_\infty$ such that $\mu(\{x\})>0$. So by Theorem~\ref{nonatomic inv measure}, the space $X_\infty$ cannot admit an ergodic compact quantum group action.
\end{proof}

\begin{corollary}\label{acc2}
For any compact quantum group action $\alpha$ on $X_\infty$, every orbit is finite.
\end{corollary}
\begin{proof}
 By Theorem~\ref{invariance of orbit}, every orbit is an $\alpha$-invariant subset on which the induced action is ergodic by Lemma~\ref{induced action on orbit is ergodic}. Notice that every orbit is also countable. Therefore it cannot be infinite by  Theorem~\ref{quantum homogenous space}.
\end{proof}

\begin{proposition}\label{every orbit in X infty is finite}
Every compact quantum group action $\alpha$ on a countable compact space is injective.
\end{proposition}
\begin{proof}
Note that $X_\infty=\bigcup_{x\in X}{\rm Orb}_x$. Corollary~\ref{acc2} says that every ${\rm Orb}_x$ is finite, so by Corollary~\ref{finite is minimal} we see that ${\rm Orb}_x$ is minimal. By Proposition~\ref{minimal invariant space is mx}, we have ${\rm Orb}_x=\mathcal{M}_x$ for every $x\in X$. So $X_\infty=\bigcup_{x\in X}\mathcal{M}_x$. By Theorem~\ref{a characterization of injectivity of alpha}, the action $\alpha$ is injective.
\end{proof}

We need the following property concerning $X_\infty$ for the next result.

\begin{lemma}\label{structure of X infty}
The subset $X_I$ is dense in $X_\infty$, hence infinite.
\end{lemma}
\begin{proof}
Let  $E=X_\infty\backslash \overline{X_I}$. We want to show that $E$ is empty. Suppose that $E$ is nonempty.
For any $x\in E$, we claim that every neighborhood of $x$ in $X_\infty$ contains a point in $E$ other than $x$. If this is not true, there exists a neighborhood $A$ of $x$ such that $A\bigcap E=\{x\}$. That is, $A\backslash \{x\}\subseteq \overline{X_I}$. Since $x\in X_A$, we have that $x\in\overline{A\backslash\{x\}}\subseteq \overline{X_I}$. This is a contradiction to $x\in E$.
Therefore $E$ is perfect, but countable. This leads to a contradiction to~\cite[Theorem 4.5]{Jech}, which says that  every perfect set in a locally compact Hausdorff space has at least the cardinality of $\mathbb{R}$.
\end{proof}
The second main result in this section is the following.
\begin{proposition}\label{acp3}
Every orbit is either contained in $X_I$, or contained in $X_A$.
\end{proposition}
\begin{proof}

We can assume the faithfulness of the Haar measure $h$ since by Proposition~\ref{orbit unchange under reduced action}, for any point $x\in X_\infty$, the orbit of $x$ under the action of $\mathcal{G}$ is the same as the orbit of $x$ under the action of $\mathcal{G}_r$.

It suffices to show that ${\rm Orb}_{x}\subseteq X_A$ for $x\in X_A$, which is equivalent to that $e_y|_{{\rm Orb}_{x}}=0$ for all $y\in X_I$ since $X_A$ is closed. Now consider the induced action $\alpha_x$ on ${\rm Orb}_{x}$. Recall that $\alpha_x(f+J)=(\pi\otimes id)\alpha(f)$ for all $f\in B$, where $J$ is the ideal consisting of functions varnishing on ${\rm Orb}_{x}$ and $\pi$ is the quotient map from $B$ onto $B/J$. Since ${\rm Orb}_{x}$ is finite, $\alpha_x$ is injective by Proposition~\ref{every orbit in X infty is finite}. So  $e_y|_{{\rm Orb}_{x}}=0$ for all $y\in X_I$ is equivalent to say that $\alpha_x(e_y)=0$ for all $y\in X_I$. Since $\alpha_x(e_y)(z)=\alpha(e_y)(z)=a_{zy}$ for all $y\in X_I$ and $z\in {\rm Orb}_{x}$,  it suffices to show that $a_{zy}=0$ for every $z$ in ${\rm Orb}_{x}$ and $y$ in $X_I$.

Let $Z'={\rm Orb}_{y}\bigcup{\rm Orb}_{z}$ for $y$ and $z$ in  $X_I$ such that ${\rm Orb}_{y}\bigcap {\rm Orb}_{z}=\emptyset$. Note that every orbit is finite and $\alpha$-invariant. So $Z'$ is $\alpha$-invariant by Proposition~\ref{union is invariant}. Consider the induced action $\alpha_{Z'}$ of $\alpha$ on $Z'$. By Proposition~\ref{Inv subset's orbit}, $Z'$ consists of only two orbits, ${\rm Orb}_{y}$ and ${\rm Orb}_{z}$. Denote $(ev_z\otimes id)\alpha_{Z'}(e_y)$ by $b_{zy}$. Note that $\alpha_{Z'}$ is an action on a finite space. By Proposition~\ref{orbits in finite spaces}, we have that $b_{zy}=0$ since $z$ and $y$ are not in the same orbit under $\alpha_{Z'}$. Observe that $b_{zy}=a_{zy}$. Therefore, by the finiteness of every orbit, after fixing $y$ in $X_I$, we get that $a_{zy}=0$ for all but finitely many $z\in X_I$. By Lemma~\ref{structure of X infty}, for every $x$ in  $X_A$, we can find a sequence $\{z_k\}_{k\geq1}\subseteq X_I$ converging to $x$. Then $a_{z_ky}=0$ for sufficiently large $k$. Therefore $a_{xy}=\lim_{k\to \infty}a_{z_ky}=0$. Since $\alpha_x$ is ergodic,
$$0=h(a_{xy})=(ev_x\otimes h)\alpha_x(e_y)=(ev_z\otimes h)\alpha_x(e_y)=h(a_{zy}),$$ for any $z\in {\rm Orb}_{x}$ and $y\in X_I$. Note that $h$ is faithful. Hence $a_{zy}=0$ for all  $z\in {\rm Orb}_{x}$ and $y\in X_I$, which completes the proof.
\end{proof}

\end{document}